\documentclass[12pt]{article}
\usepackage{bbm}
\usepackage{amsfonts}
\usepackage{pifont}
\usepackage{hyperref}
\usepackage{mathrsfs}
\usepackage{indentfirst}
\usepackage{amsmath}
\usepackage{amssymb}
\usepackage[amsmath, thmmarks]{ntheorem}
\usepackage{cite}
\usepackage{graphicx}
\usepackage{tikz}
\newtheorem{definition}{\bf Definition}[section]
\newtheorem{lemma}{\bf Lemma}[section]
\newtheorem{theorem}{\bf Theorem}[section]
\newtheorem{remark}{\bf Remark}[section]

\newtheorem{proposition}{\bf Proposition}[section]
\newtheorem{example}{\bf Example}[section]

\begin{document}
\setcounter{page}{1}

\title{{\textbf{Uninorms via two comparable closure operators on bounded lattices}}\thanks {Supported by National Natural Science
Foundation of China (No.11871097, 12271036)}}
\author{Zhenyu Xiu$^{a}$, Xu Zheng$^{b}$\footnote{Corresponding author.
\emph{E-mail address}: xyz198202@163.com (Z.Y. Xiu), 3026217474@qq.com (X. Zheng). }    \\
\emph{\small $^{a}$College of Applied Mathematics, Chengdu University of Information Technology, }\\
\emph{\small Chengdu 610000, PR China }\\
\emph{\small $^{b}$College of Mathematics and Statistics, Northwest Normal University, }\\
\emph{\small Lanzhou 730070, PR China }
}

\newcommand{\pp}[2]{\frac{\partial #1}{\partial #2}}
\date{}
\maketitle

\begin{quote}
{\bf Abstract}
In this paper, we propose novel methods for constructing uninorms using two comparable closure operators or, alternatively, two comparable interior operators on bounded lattices. These methods are developed under the necessary and sufficient conditions imposed on these operators.  Specifically, the construction of uninorms for  $(x ,y )\in ]0 ,e [\times]e ,1 [ \cup  ]e ,1 [\times]0 ,e [$  depends not only on the structure of the bounded lattices but also on the chosen closure operators (or interior operators). Consequently, the resulting uninorms do not necessarily belong to  $\mathcal{U}_{min}^{*}\cup \mathcal{U}_{min}^{1}$ (or $\mathcal{U}_{max}^{*}\cup\mathcal{U}_{max}^{0}$).
Moreover, we present the degenerate cases of the aforementioned results, which are constructed using only a single closure operator or a single interior operator. Some of these cases correspond to well-known results documented in the literature.



\end{quote}

\section{Introduction}\label{intro}

In 1996, Yager and Rybalov \cite{RR96} introduced the uninorms on the unit interval $[0,1]$ as a generalization of triangular norms (t-norms, for short) and triangular conorms (t-conorms, for short) \cite{KM42}, which allow a neutral element $e$ to lie anywhere in $[0,1]$ rather than at $1$ or $0$. Since then, uninorms have been proven to have applications in several fields, such as fuzzy logic, fuzzy set theory, fuzzy system modeling, expert systems, neural networks and so on (see, e.g., \cite{BD99,MG09,MG11,WP07,RR02,GM07}).  In particular, uninorms
are important in bipolar decision making, e.g. already mentioned in expert systems (see e.g., \cite{BD99,PH85,PH92,MT08}).
 Moreover, they  play the role of multiplication in the domain of integration \cite{EP10}.

In 2015, Kara\c{c}al and Mesiar \cite{FK15} generalized the concept of uninorms from the unit interval $[0,1]$
to bounded lattices and introduced some construction methods for uninorms on
bounded lattices.
 Since then, numerous construction methods have been proposed in the literature. The constructions of uninorms are typically based on  the following tools, such as t-norms (or, respectively, t-conorms) (see, e.g., \cite{FK15,EA21,SB14,GD16,GD17,GD19,GD19.1,GD19.2,GD20,GD22,YD19, YD20,FK17,AX20,ZY23}), t-subnorms (or, respectively,   t-superconorms) (see, e.g.,  \cite{ZY23,XJ20,WJ21,HP21}), closure operators (or, respectively,  interior operators) (see, e.g., \cite{GD21,GD23.1,GD23.2,GD24,XJ21,YO20,BZ21}),  additive generators \cite{HeP} and uninorms (see, e.g., \cite{GD23,ZX23,ZX24}).


Closure operators on the powerset $P(X)$  of a nonempty set  $X$     have long been fundamental concepts in topology and are widely used as tools for constructing topologies on $X$ \cite{R.E}.  In fact, closure (interior) operators on   $P(X)$    are fundamentally defined based on the inherent lattice structure of  $P(X)$,  where set inclusion serves as the partial order, and set intersection and set union act as the meet and join, respectively.  In \cite{CJ44},closure (interior) operators on the inherent lattice structure on $P(X)$
 can be readily extended to the lattice  by relaxing the axiom that the closure (interior)  of $\emptyset  (X)$ is $\emptyset  (X)$.




In this paper, we primarily focus on construction methods for uninorms on  arbitrary bounded lattices using two closure operators  (or, respectively, two interior operators).  First, we analyze the existing construction methods using closure operators (or, respectively, interior operators) in the literature and then provide the following observations: (1), (2) and (3).


(1) Based on the additional constraints, the results can be classified as follows:
\begin{itemize}
   \item   No additional constraints.  (see, e.g., \cite{GD23.1,YO20})
\item Sufficient conditions on bounded lattices. (see, e.g., Theorems 3.4 and 3.12 in \cite{GD21},  Propositions 3.6 and 3.7, Theorems  3.3 and 3.4, and  Corollaries 4.3, 4.4, 4.5 and 4.6 in \cite{BZ21})
\item Sufficient and necessary conditions on bounded lattices. (see, e.g., Theorems 3.1, 3.4, 3.10 and 3.12 in \cite{GD21},  Theorems 3.1 and 3.11 in \cite{GD23.2}, Theorems 1 and 2 in \cite{GD24},  Theorems 4.1 and 4.4  in \cite{XJ21}, Propositions  3.5 and 3.6, Theorem 3.9, and  Corollaries 4.2, 4.4 and 4.7 in \cite{BZ21} )
\item  Sufficient conditions on  closure operators (or, respectively,  interior operators). (see, e.g.,  Theorems  3.4 and 3.12 in \cite{GD21}, Propositions 3.1 and 3.5, Theorem 3.9(2), and Corollaries 4.1, 4.2 and 4.7(2) in \cite{BZ21})
\item  Sufficient and necessary conditions on closure operators (or, respectively,  interior operators). (see, e.g., Theorem 3.9(1) and Corollary 4.7(1) in \cite{BZ21})
\item  Sufficient conditions on  t-norms. (see, e.g., Propositions 3.7, and Corollaries 4.5 and 4.6 in \cite{BZ21})
\item  Sufficient conditions on t-conorms. (see, e.g., Theorem 3.4, Proposition 3.7  and Corollary 4.6 in \cite{BZ21})
\end{itemize}

(2) All the aforementioned uninorms belong to $\mathcal{U}_{max}^{*}\cup\mathcal{U}_{max}^{0} $ or $\mathcal{U}_{min}^{*}\cup \mathcal{U}_{min}^{1}$.


(3) The construction of all the aforementioned uninorms on  $]0 ,e [\times]e ,1 [  \cup   ]e ,1 [\times]0 ,e [$  depends  on the given bounded lattices rather than on closure operators  or interior operators.
 This is because the value of uninorms for $(x ,y )\in ]0 ,e [\times]e ,1 [ \cup  ]e ,1 [\times]0 ,e [$
 is always determined as $x $ or $y$.


Based on the above remarks, we can explore the construction of uninorms for   $x \in ]0 ,e [\cup I_{e }$ and $y \in]e ,1 [$ (or $x \in ]e ,1 [\cup I_{e }$ and $y \in]0 ,e [$) using two comparable closure operators (or, respectively,  two comparable interior operators). These construction methods differ from all previously established methods involving closure operators (or interior operators) in the literature.
Moreover,  the construction  of   new uninorms    for
$(x ,y )\in ]0 ,e [\times]e ,1 [ \cup  ]e ,1 [\times]0 ,e [$
depends not only on the given bounded lattices but also on closure operators   (or, respectively,  interior operators).   Consequently,  the resulting uninorms  do not necessarily  belong to  $\mathcal{U}_{min}^{*}\cup \mathcal{U}_{min}^{1}$  (or, respectively, $\mathcal{U}_{max}^{*}\cup\mathcal{U}_{max}^{0}$).  Furthermore, we aim to demonstrate that the additional constraints on the two comparable closure operators (or, respectively,  two comparable interior operators) are both necessary and sufficient.


The remainder of this paper is organized as follows. In Section 2, we recall some concepts and results essential for this manuscript.
In Section 3, we present two new construction methods for uninorms on a bounded lattice $L$ with a t-conorm on  $[e ,1 ]$  and two comparable closure operators on $L$ under the necessary and sufficient conditions on them.  Meanwhile,  the dual results are also presented. Moreover, we give the degeneration of all the above results based on two comparable closure operators (or, respectively, interior  operators).
In Section 4, we conclude with some remarks.


\section{Preliminaries}


In this section, we review some fundamental concepts and results related to lattices and aggregation functions.
\begin{definition}{\rm\cite{GB67}}\label{de2.1}
A lattice $(L,\leq)$  is called bounded if there exist two elements $1,0\in L$  such that $0\leq x\leq 1$ for all $x\in L$.
\end{definition}

Throughout this article, unless otherwise stated, let $L$ denote a bounded lattice with top and bottom elements $1$ and $0$, respectively.
\begin{definition}{\rm\cite{GB67}}\label{de2.2}
Let $L$ be a bounded lattice, $a ,b \in L$ with $a \leq b $. A subinterval $[a ,b ]$ of $L$ is defined as
$[a , b ]=\{x \in L: a \leq x  \leq b \}.$
Similarly, we can define $[a , b [=\{x \in L: a \leq x  < b \}, ]a , b ]=\{x \in L: a < x  \leq b \}$ and $]a , b [=\{x \in L: a< x  < b\}$.
\end{definition}

In the following, we use the notation $a\parallel b$ to represent that $a$  and $b$  are incomparable. For each $a\in L$, let $I_{a }$ denote the set of all  elements in $L$ that are incomparable with $a$, that is, $I_{a }=\{x \in L\mid x \parallel a\}$.

\begin{definition}{\rm\cite{SS06}}\label{de2.3}
An operation $T :L^{2}\rightarrow L$  is called a t-norm on $L$ if it is commutative, associative, and increasing with respect to both variables, and it has the neutral element $1 \in L$, that is, $T (1 ,x )=x $  for all $x \in L$.
\end{definition}

\begin{definition}{\rm\cite{GD16}}\label{}
An operation $S :L^{2}\rightarrow L$  is called a t-conorm on  $L$ if it is commutative, associative, and increasing with respect to both variables, and it has the neutral element $0 \in L$, that is, $S (0 ,x )=x $  for all $x \in L$.
\end{definition}

\begin{definition}{\rm\cite{CJ44}}\label{de29}
A mapping $cl : L\rightarrow L$ is said to be a closure operator on $L$ if it satisfies the following three conditions:\\
{\rm(CL1)} $x \leq cl (x )$ for any $x\in L$;\\
{\rm(CL2)}  $cl (x \vee y )= cl (x )\vee cl (y )$ for any $x,y\in L$;\\
{\rm(CL3)}  $cl (cl (x ))= cl (x )$ for any $x\in L$.
\end{definition}

\begin{proposition}
	Let $cl$ be a closure operator on $L$. Then it satisifies\\
{\rm(CL4)}  $x\leq y$ implies $cl(x)\leq cl(y)$ for any $x,y\in L$.
\end{proposition}
\begin{proof}
	It follows immediately from (CL2).
\end{proof}

\begin{definition}{\rm\cite{YO20}}\label{de30}
A mapping $int : L\rightarrow L$ is said to be an interior operator on $L$ if it satisfies the following three conditions:\\
{\rm(IN1)}  $int (x )\leq x $ for any $x\in L$;\\
{\rm(IN2)}  $int (x \wedge y )= int (x )\wedge int (y )$ for any $x ,y \in L$;\\
{\rm(IN3)}  $int (int (x ))= int (x )$ for any $x\in L$.
\end{definition}

\begin{proposition}
	Let $int$ be an interior operator on $L$. Then it satisifies\\
	{\rm(IN4)}  $x\leq y$ implies $int(x)\leq int(y)$ for any $x,y\in L$.
\end{proposition}
\begin{proof}
	It follows immediately from (IN2).
\end{proof}

\begin{definition}{\rm\cite{FK15}}\label{de2.4}
An operation $U :L^{2}\rightarrow L$  is called a uninorm on  $L$ if it is commutative, associative, and increasing with respect to both variables, and it has the neutral element $e \in L$, that is, $U (e ,x )=x $  for all $x \in L$.
\end{definition}


For convenience, we present the following notations in a bounded lattice $(L,\leq,0 ,1 )$ with  $e \in L\setminus\{0 ,1 \}$, which can be found in \cite{HP21}.
\begin{itemize}
	\item $\mathcal{U}_{min}$: The class of all uninorms $U $ on $L$ with neutral element $e $ satisfying the following condition: $U (x ,y )=y $, for all $(x ,y )\in]e ,1 ]\times L\setminus[e ,1 ]$.
	\item $\mathcal{U}_{max}$: The class of all uninorms $U $ on $L$ with neutral element $e $ satisfying the following condition:
	$U (x ,y )=y $, for all $(x ,y )\in[0 ,e [\times L\setminus[0 ,e ]$.
	\item $\mathcal{U}_{min}^{*}$: The class of all uninorms $U $ on $L$ with neutral element $e $ satisfying the following condition: $U (x ,y )=y $ for all $(x ,y )\in]e ,1 ]\times [0 ,e [$.
	\item 	$\mathcal{U}_{max}^{*}$: The class of all uninorms $U $ on $L$ with neutral element $e $ satisfying the following condition: $U (x ,y )=y $ for all $(x ,y )\in[0 ,e [\times ]e ,1 ]$.
	\item 	$\mathcal{U}_{min}^{r}$: The class of all uninorms $U $ on $L$ with neutral element $e $ satisfying the following condition: $U (x ,y )=x $ for all $(x ,y )\in]e ,1 ]\times L\setminus[e ,1 ]$.
	\item 	$\mathcal{U}_{max}^{r}$: The class of all uninorms $U $ on $L$ with neutral element $e $ satisfying the following condition: $U (x ,y )=x $ for all $(x ,y )\in[0 ,e [\times L\setminus[0 ,e ]$.
	\item $\mathcal{U}_{min}^{1}$: The class of all uninorms $U $ on $L$ with neutral element $e $ satisfying the following two conditions: $U (x ,y )=y $, for all $(x ,y )\in]e ,1 [\times L\setminus[e ,1 ]$ and $U (1 ,y )=1 $, for all $y \in L\setminus[e ,1 ]$.
	\item 	$\mathcal{U}_{max}^{0}$: The class of all uninorms $U $ on $L$ with neutral element $e $ satisfying the following two conditions: $U (x ,y )=y $, for all $(x ,y )\in]0 ,e [\times L\setminus[0 ,e ]$ and $U (0 ,y )=0 $, for all $y \in L\setminus[0 ,e ]$.
\end{itemize}

\begin{remark}{\rm\cite{HP21}}
 $\mathcal{U}_{max}\cup \mathcal{U}_{min}^{r}\subseteq \mathcal{U}_{max}^{*}$ and $\mathcal{U}_{min}\cup \mathcal{U}_{max}^{r}\subseteq \mathcal{U}_{min}^{*}$.
\end{remark}

\begin{proposition}{\rm\cite{WJ21}}\label{pro2.1}
Let $H$ be a commutative binary operation on $S$ and $A_{1},A_{2},\ldots,A_{n}$ be subsets of $S$. Then $H$ is associative on $A_{1}\cup A_{2}\cup\ldots\cup A_{n}$ if and only if the following statements hold:\\
$(i)$ for every combination $\{i,j,k\}$ of size $3$ chosen from $\{1,2,\ldots,n\}$, $H(x_{i},H(y_{j},z_{k}))=H(H(x_{i},y_{j}),z_{k})=H(y_{j},H(x_{i},z_{k}))$ for all $x_{i}\in A_{i},y_{j}\in A_{j},z_{k}\in A_{k}$;\\
$(ii)$ for every combination $\{i,j\}$ of size $2$ chosen from $\{1,2,\ldots,n\}$, $H(x_{i},H(y_{i},z_{j}))=H(H(x_{i},y_{i}),z_{j}) $ for all $x_{i}\in A_{i},y_{i}\in A_{i},z_{j}\in A_{j}$;\\
$(iii)$ for every combination $\{i,j\}$ of size $2$ chosen from $\{1,2,\ldots,n\}$, $H(x_{i},H(y_{j},z_{j}))=H(H(x_{i},y_{j}),z_{j}) $ for all $x_{i}\in A_{i},y_{j}\in A_{j},z_{j}\in A_{j}$;\\
$(iv)$ for every $i\in \{1,2,\ldots,n\}$, $H(x_{i},H(y_{i},z_{i}))=H(H(x_{i},y_{i}),z_{i}) $ for all $x_{i},y_{i},z_{i}$ $\in A_{i}$.
\end{proposition}

%

\section{New methods to construct uninorms on bounded lattices}
\label{}

In this section, based on two comparable  closure operators,  we propose some new construction methods for uninorms on arbitrary bounded lattices
 with a given t-conorm $S$ on the subinterval $[e ,1 ]$  of $L$.
 These new uninorms on bounded lattices need  not belong to  $\mathcal{U}_{min}^{*}\cup \mathcal{U}_{min}^{1}$.   Moreover,  the dual results are also presented.
 Meanwhile, illustrative examples and figures for the construction of uninorms on bounded lattices are provided.

First, let us explore closure operators and interior operators on bounded lattices.

\begin{lemma}\label{le2}
Let $cl $ be a closure operator on $L$. Then $cl (cl (x )\wedge y )=cl (x )$ for any $x \leq y $.
\end{lemma}
\begin{proof}
Take any $x,y\in L$ such that $x\leq y$. On one hand, it follows from (CL1) that $x \leq cl (x )\wedge y $. Then by (CL4), we have $cl (x )\leq cl (cl (x )\wedge y )$. On the other hand, since $cl (x )\wedge y \leq cl (x )$, it follows from (CL3) and (CL4) that $cl (cl (x )\wedge y )\leq cl (cl (x ))=cl (x )$. Thus, $cl (cl (x )\wedge y )=cl (x )$ for any $x\leq y$.
\end{proof}

\begin{lemma}\label{}
Let $int $ be an interior operator on $L$. Then $int (int (x )\vee y )=int (x )$ for any $y \leq x $.
\end{lemma}
\begin{proof}
It can be proved immediately by the proof similar to Lemma \ref{le2}.
\end{proof}\


Now, we present a new construction method for uninorms on a bounded lattice using a t-conorm on $[e, 1]$ and two closure operators on $L$.

\begin{theorem}\label{th32}
Let $e \in L\setminus\{0 ,1 \}$,
$S $ be a t-conorm on $[e ,1 ]$,  and $cl_{1}$ and $cl_{2}$ be closure operators on $L$ with $cl_{1}(x )\leq cl_{2}(x )$ for all $x \in L\setminus[e ,1 ]$. Define $U :L^{2}\rightarrow L$  by
\begin{eqnarray*}
U (x ,y )=\left\{
\begin{array}{ll}
S (x , y ) & $if$\ (x ,y )\in [e ,1 ]^{2},\\
x  & $if$\ (x ,y )\in (I_{e }\cup ]0 ,e [)\times \{e \},\\
y  & $if$\ (x ,y )\in \{e \}\times (I_{e }\cup ]0 ,e [),\\
cl_{1}(x )\wedge (x \vee e ) & $if$\ (x ,y )\in  ]0 ,e [  \times ]e ,1 ],\\
cl_{1}(y )\wedge (y \vee e ) & $if$\ (x ,y )\in ]e ,1 ]   \times ]0 ,e [,\\
cl_{2}(x )\wedge (x \vee e ) & $if$\ (x ,y )\in I_{e }\times ]e ,1 ],\\
cl_{2}(y )\wedge (y \vee e ) & $if$\ (x ,y )\in ]e ,1 ] \times  I_{e },\\
0  &  otherwise.
\end{array} \right.
\end{eqnarray*}
Then $U $ is a uninorm on $L$ with the neutral element $e$ if and only if $cl_{1}(x )\notin [e ,1 ]$ for all $x \in ]0 ,e [$ and $cl_{2}(x )\notin [e ,1 ]$ for all $x \in I_{e }$.
\end{theorem}
\begin{proof}
Necessity: Suppose that $U$ is a uninorm on $L$ with the neutral element $e $.

First, we prove that $cl_{1}(x )\notin [e ,1 ]$ for all $x \in ]0 ,e [$. Assume that there exists $x \in ]0 ,e [$ such that $cl_{1}(x )\in [e ,1 ]$. Then $U (U (1 ,x ),x )=U (cl_{1}(x )\wedge (x \vee e ),x )=U (e ,x )=x $ and $U (1 ,U (x ,x ))=U (1 ,0 )=0 $. Since $x \neq 0 $, this contradicts the associativity property of $U $. Thus $cl_{1}(x )\notin [e ,1 ]$ for all $x \in ]0 ,e [$.

Next, we prove that $cl_{2}(x )\notin [e ,1 ]$ for all $x \in I_{e }$.  By Definition \ref{de29},  we know that $cl_{2}(x )\neq e $ for all $x \in I_{e }$. In the following, we just prove that $cl_{2}(x )\notin ]e ,1 ]$ for all $x \in I_{e }$. Assume that there exists $x \in I_{e }$ such that $cl_{2}(x )\in ]e ,1 ]$. Then $cl_{2}(x )\wedge (x \vee e )\in ]e ,1 ]$. Therefore $U (U (1 ,x ),x )=U (cl_{2}(x )\wedge (x \vee e ),x )=cl_{2}(x )\wedge (x \vee e )$ and $U (1 ,U (x ,x ))=U (1 ,0 )=0 $. Since $cl_{2}(x )\wedge (x \vee e )\neq 0 $, this contradicts the associativity property of $U $. Thus $cl_{2}(x )\notin [e ,1 ]$ for all $x \in I_{e }$.

Sufficiency: It is easy to see that $U $ is commutative and  $e$  is a neutral element. So it remains to show the increasingness and associativity of $U$.

Increasingness:  Take any $x,y,z\in L$ such that $x \leq y $. It follows immediately that $U (x ,z )\leq U (y ,z )$ if $0 \in \{x ,y ,z \}$ or both $x $ and $y $ belong to one of the  intervals $]0 ,e [, \{e \}, I_{e }$ or $]e ,1 ]$. The residual proof can be split into all possible cases:

1. $x \in ]0 ,e [$

\ \ \ 1.1. $y \in \{e \}$

\ \ \ \ \ \ 1.1.1. $z \in ]0 ,e [\cup I_{e }$

\ \ \ \ \ \ \ \ \ \ \ \ $U (x ,z )=0 < z =U (y ,z )$

\ \ \ \ \ \ 1.1.2. $z \in \{e \}$

\ \ \ \ \ \ \ \ \ \ \ \ $U (x ,z )=x < y =S (y ,z )=U (y ,z )$

\ \ \ \ \ \ 1.1.3. $z \in ]e ,1 ]$

\ \ \ \ \ \ \ \ \ \ \ \ $U (x ,z )=cl_{1}(x )\wedge (x \vee e )=cl_{1}(x )\wedge e <e <z =S (y ,z )=U (y ,z )$

\ \ \ 1.2. $y \in I_{e }$

\ \ \ \ \ \ 1.2.1. $z \in ]0 ,e [\cup I_{e }$

\ \ \ \ \ \ \ \ \ \ \ \ $U (x ,z )=0 =U (y ,z )$

\ \ \ \ \ \ 1.2.2. $z \in \{e \}$

\ \ \ \ \ \ \ \ \ \ \ \ $U (x ,z )=x < y =U (y ,z )$

\ \ \ \ \ \ 1.2.3. $z \in ]e ,1 ]$

\ \ \ \ \ \ \ \ \ \ \ \ $U (x ,z )=cl_{1}(x )\wedge (x \vee e )\leq cl_{2}(y )\wedge (y \vee e )=U (y ,z )$

\ \ \ 1.3. $y \in ]e ,1 ]$

\ \ \ \ \ \ 1.3.1. $z \in ]0 ,e [$

\ \ \ \ \ \ \ \ \ \ \ \ $U (x ,z )=0 < cl_{1}(z )\wedge (z \vee e )=U (y ,z )$

\ \ \ \ \ \ 1.3.2. $z \in \{e \}$

\ \ \ \ \ \ \ \ \ \ \ \ $U (x ,z )=x  < y =S (y ,z )=U (y ,z )$

\ \ \ \ \ \ 1.3.3. $z \in I_{e }$

\ \ \ \ \ \ \ \ \ \ \ \ $U (x ,z )=0 <cl_{2}(z )\wedge (z \vee e )=U (y ,z )$

\ \ \ \ \ \ 1.3.4. $z \in ]e ,1 ]$

\ \ \ \ \ \ \ \ \ \ \ \ $U (x ,z )=cl_{1}(x )\wedge (x \vee e )=cl_{1}(x )\wedge e <e < y \leq S (y ,z )=U (y ,z )$

2. $x \in \{e \}, y \in ]e ,1 ]$

\ \ \ 2.1. $z \in ]0 ,e [$

\ \ \ \ \ \ \ \ \ \ \ \ $U (x ,z )=z \leq cl_{1}(z )\wedge (z \vee e )=U (y ,z )$

\ \ \ 2.2. $z \in [e ,1 ]$

\ \ \ \ \ \ \ \ \ \ \ \ $U (x ,z )=S (x ,z )\leq S (y ,z )=U (y ,z )$

\ \ \ 2.3. $z \in I_{e }$

\ \ \ \ \ \ \ \ \ \ \ \ $U (x ,z )=z \leq  cl_{2}(z )\wedge (z \vee e )=U (y ,z )$

3. $x \in I_{e }, y \in ]e ,1 ]$

\ \ \ 3.1. $z \in ]0 ,e [$

\ \ \ \ \ \ \ \ \ \ \ \ $U (x ,z )=0 < cl_{1}(z )\wedge (z \vee e )=U (y ,z )$

\ \ \ 3.2. $z \in \{e \}$

\ \ \ \ \ \ \ \ \ \ \ \ $U (x ,z )=x < y =S (y ,z )=U (y ,z )$

\ \ \ 3.3. $z \in I_{e } $

\ \ \ \ \ \ \ \ \ \ \ \ $U (x ,z )=0 <cl_{2}(z )\wedge (z \vee e )= U (y ,z )$

\ \ \ 3.4. $z \in  ]e ,1 ]$

\ \ \ \ \ \ \ \ \ \ \ \ $U (x ,z )=cl_{2}(x )\wedge (x \vee e )< x \vee e \leq y \leq  S (y ,z )= U (y ,z )$

Based on the above cases, we always obtain that $U (x ,z )\leq U (y ,z )$ for all $x ,y ,z \in L$ with $x \leq y $. Therefore, $U $ is increasing.

Associativity: First,  we show the following two facts.

(1) If $cl_{1}(x )\notin [e ,1 ]$ for all $x \in ]0 ,e [$, then $cl_{1}(x )\wedge (x \vee e )\in ]0 ,e [$ for all $x \in ]0 ,e [$.

Take any $x \in ]0 ,e [$. On one hand, since $x \leq cl_{1}(x )$ and $x < x \vee e $,
$cl_{1}(x )\wedge (x \vee e ) \geq  x  > 0 $.
On the other hand, since $cl_{1}(x )\notin [e ,1 ]$ and $x \vee e =e $, $cl_{1}(x )\wedge (x \vee e )<e $. Thus $cl_{1}(x )\wedge (x \vee e )\in ]0 ,e [$.

(2) If $cl_{2}(x )\notin [e ,1 ]$ for all $x \in I_{e }$, then $cl_{2}(x )\wedge (x \vee e )\in I_{e }$ for all $x \in I_{e }$.

Take any $x \in I_{e }$. Since $cl_{2}(x )\notin [e ,1 ]$, $cl_{2}(x )\wedge (x \vee e )\notin [e ,1 ]$. Next, we show that $cl_{2}(x )\wedge (x \vee e )\notin [0 ,e [$. If $cl_{2}(x )\wedge (x \vee e )\in [0 ,e [$, then $cl_{2}(x )\wedge (x \vee e )<x $ or $cl_{2}(x )\wedge (x \vee e )\parallel x $. This contradicts $x \leq cl_{2}(x )$ and $x < x \vee e $. Thus $cl_{2}(x )\wedge (x \vee e )\notin [0 ,e [$.  Therefore $cl_{2}(x )\wedge (x \vee e )\in I_{e }$.


Next, we prove  that $U (x ,U (y ,z ))=U (U (x ,y ),z )$  for all $x ,y ,$ $z \in L$. It is obvious that $U (x , U (y ,z ))=U (U (x ,y ),z )=U (y , U (x ,z ))$ if $0 \in \{x ,y ,z \}$ or  $e \in \{x ,y ,z \}$.
By Proposition \ref{pro2.1}, we need to verify the following cases:

1. If $x ,y ,z \in ]0 ,e [$, then $U (x , U (y ,z ))=U (x , 0 )=0 =U (0 ,z )=U (U (x ,y ),z )$.

2. If $x ,y ,z \in I_{e }$, then $U (x , U (y ,z ))=U (x , 0 )=0 =U (0 ,z )=U (U (x ,y ),z )$.

3. If $x ,y ,z \in ]e ,1 ]$, then $U (x , U (y ,z ))=U (x , S (y ,z ))=S (x , S (y ,z ))=S (S (x ,y ),z )$ $=U (S (x ,y ),z )=U (U (x ,$ $y ),z )$.

4. If $x ,y \in ]0 ,e [$, $z \in I_{e }$, then $U (x , U (y ,z ))=U (x ,0 )=0 =U (0 ,z )=U (U (x ,y ),z )$.

5. If $x ,y \in ]0 ,e [$ and $z \in ]e ,1 ]$, then $U (x , U (y ,z ))=U (x ,$ $cl_{1}(y )\wedge (y \vee e ))=0 =U (0 ,z )=U (U (x ,y ),z )$.

6. If $x ,y \in I_{e }$ and $z \in ]e ,1 ]$, then $U (x , U (y ,z ))=U (x ,$ $cl_{2}(y )\wedge (y \vee e ))=0 =U (0 ,z )=U (U (x ,y ),z )$.

7. If $x \in ]0 ,e [$ and $y ,z \in I_{e }$, then $U (x , U (y ,z ))=U (x ,0 )$ $=0 =U (0 ,z )=U (U (x ,y ),z )$.

8. If $x \in ]0 ,e [$ and $y ,z \in ]e ,1 ]$, then $U (x , U (y ,z ))=U (x ,$ $S (y ,z ))=cl_{1}(x )\wedge(x \vee e )$ and $U (U (x ,y ),z )=U (cl_{1}(x )\wedge(x \vee e ),z )=cl_{1}(cl_{1}(x )\wedge(x \vee e ))\wedge ((cl_{1}(x )\wedge(x \vee e ))\vee e )$.
By Lemma \ref{le2}, we can obtain that $cl_{1}(x )=cl_{1}(cl_{1}(x )\wedge(x \vee e ))$ and $x \vee e =(cl_{1}(x )\wedge(x \vee e ))\vee e $. Thus $U (x , U (y ,z ))=U (U (x ,y ),z )$.

9. If $x \in I_{e }$ and $y ,z \in ]e ,1 ]$, then $U (x , U (y ,z ))=U (x ,$ $S (y ,z ))=cl_{2}(x )\wedge(x \vee e )$ and $U (U (x ,y ),z )=U (cl_{2}(x )\wedge(x \vee e ),z )=cl_{2}(cl_{2}(x )\wedge(x \vee e ))\wedge ((cl_{2}(x )\wedge(x \vee e ))\vee e )$.
By Lemma \ref{le2}, we can obtain that $cl_{2}(x )=cl_{2}(cl_{2}(x )\wedge(x \vee e ))$ and $x \vee e =(cl_{2}(x )\wedge(x \vee e ))\vee e $. Thus $U (x , U (y ,z ))=U (U (x ,y ),z )$.

10. If $x \in ]0 ,e [$, $y \in I_{e }$ and $z \in ]e ,1 ]$, then $U (x , U (y ,z ))=U (x ,cl_{2}(y )\wedge(y \vee e ))=0 =U (0 ,z )=U (U (x ,y ),z )$ and $U (y , U (x ,z ))=U (y ,cl_{1}(x )\wedge(x \vee e ))=0 $. Thus $U (x , U (y ,z ))$ $=U (U (x ,y ),z )=U (y , U (x ,z ))$.

Thus, the associative property is obtained by combining Proposition \ref{pro2.1} and the commutative property of $U$.

Therefore, $U $ is a uninorm on $L$ with the neutral element $e $.
\end{proof}

\begin{remark}\label{rema21}
Note that the uninorm $U :L^{2}\rightarrow L$ in Theorem \ref{th32} can be also defined in the following way:
\begin{eqnarray*}
U (x ,y )=\left\{
\begin{array}{ll}
S (x , y ) & $if$\ (x ,y )\in ]e ,1 ]^{2},\\
e  & $if$\ (x ,y )\in \{e \}\times \{e \},\\
x  & $if$\ (x ,y )\in (L\setminus\{0 ,e \}) \times \{e \},\\
y  & $if$\ (x ,y )\in \{e \}\times (L\setminus\{0 ,e \}),\\
cl_{1}(x )\wedge (x \vee e ) & $if$\ (x ,y )\in  ]0 ,e [\times ]e ,1 ],\\
cl_{1}(y )\wedge (y \vee e ) & $if$\ (x ,y )\in ]e ,1 ]\times ]0 ,e [,\\
cl_{2}(x )\wedge (x \vee e ) & $if$\ (x ,y )\in I_{e }\times ]e ,1 ],\\
cl_{2}(y )\wedge (y \vee e ) & $if$\ (x ,y )\in ]e ,1 ]\times I_{e },\\
0  & $if$\ (x ,y )\in \{0 \}\times L \cup L\times\{0 \} \cup I_{e }\times I_{e }\cup (I_{e }\cup ]0 ,e [)\times]0 ,e [ \\
  & \ \ \  \cup ]0 ,e [\times(I_{e }\cup ]0 ,e [).
\end{array} \right.
\end{eqnarray*}

\end{remark}

\begin{remark}
From Remark \ref{rema21}, the structure of the uninorm $U :L^{2}\rightarrow L$ is illustrated in Fig.1., where $ \ast $ denotes $ cl_{1}(x )\wedge (x \vee e ), \star $ denotes $ cl_{1}(y )\wedge (y \vee e ), \circ $ denotes $ cl_{2}(x )\wedge (x \vee e )$ and $ \bullet $ denotes $ cl_{2}(y )\wedge (y \vee e ).$
\end{remark}

The next example illustrates the construction method of uninorms on bounded lattices in Theorem \ref{th32}.
\begin{example}\label{ex1}
Given a bounded lattice $L_{1}=\{0 ,a ,b ,e ,m ,k ,s ,n ,j ,1 \}$ depicted in Fig.2., a t-conorm $S $ on $[e ,1 ]$ defined by $S (x ,y )=x \vee y $ for all $x ,y \in [e ,1 ]$ and closure operators  $cl_{1}$ and $cl_{2}$ on $L_{1}$ shown in  Table \ref{Tab:89}.
It is easy to see that the closure operators $cl_{1}$ and $cl_{2}$ on $L_{1}$ satisfy the conditions in Theorem \ref{th32}, i.e., $cl_{1}(x )\leq cl_{2}(x )$ for all $x \in L_{1}$, $cl_{1}(x )\notin [e ,1 ]$ for all $x \in ]0 ,e [$ and $cl_{2}(x )\notin [e ,1 ]$ for all $x \in I_{e }$. Using the construction method in Theorem \ref{th32}, we can obtain a uninorm $U :L_{1}^{2}\rightarrow L_{1}$ with the neutral element $e$, as shown in Table \ref{Tab:01}.
\end{example}

\begin{minipage}{11pc}
\setlength{\unitlength}{0.75pt}\begin{picture}(600,230)
\put(158,25){\makebox(0,0)[l]{\footnotesize$]0 ,e [$}}
\put(72,73){\makebox(0,0)[l]{\footnotesize$]0 ,e [$}}
\put(72,130){\makebox(0,0)[l]{\footnotesize$]e ,1 ]$}}
\put(258,25){\makebox(0,0)[l]{\footnotesize$]e ,1 ]$}}
\put(205,25){\makebox(0,0)[l]{\footnotesize$\{e \}$}}
\put(366,25){\makebox(0,0)[l]{\footnotesize$I_{e }$}}
\put(80,103){\makebox(0,0)[l]{\footnotesize$\{e \}$}}
\put(80,43){\makebox(0,0)[l]{\footnotesize$\{0 \}$}}
\put(105,25){\makebox(0,0)[l]{\footnotesize$\{0 \}$}}
\put(90,193){\makebox(0,0)[l]{\footnotesize$I_{e }$}}
\put(168,73){\makebox(0,0)[l]{\footnotesize$0 $}}
\put(168,103){\makebox(0,0)[l]{\footnotesize$x $}}
\put(168,130){\makebox(0,0)[l]{\footnotesize$\ast$}}
\put(168,193){\makebox(0,0)[l]{\footnotesize$0 $}}
\put(118,73){\makebox(0,0)[l]{\footnotesize$0 $}}
\put(118,103){\makebox(0,0)[l]{\footnotesize$0 $}}
\put(118,130){\makebox(0,0)[l]{\footnotesize$0 $}}
\put(118,193){\makebox(0,0)[l]{\footnotesize$0 $}}
\put(218,73){\makebox(0,0)[l]{\footnotesize$y $}}
\put(218,103){\makebox(0,0)[l]{\footnotesize$e $}}
\put(218,130){\makebox(0,0)[l]{\footnotesize$y $}}
\put(218,193){\makebox(0,0)[l]{\footnotesize$y $}}
\put(268,73){\makebox(0,0)[l]{\footnotesize$\star$}}
\put(268,103){\makebox(0,0)[l]{\footnotesize$x $}}
\put(250,130){\makebox(0,0)[l]{\footnotesize$S (x ,y )$}}
\put(268,193){\makebox(0,0)[l]{\footnotesize$\bullet$}}
\put(368,73){\makebox(0,0)[l]{\footnotesize$0 $}}
\put(368,103){\makebox(0,0)[l]{\footnotesize$x $}}
\put(368,130){\makebox(0,0)[l]{\footnotesize$\circ$}}
\put(368,193){\makebox(0,0)[l]{\footnotesize$0 $}}

\put(118,43){\makebox(0,0)[l]{\footnotesize$0 $}}
\put(168,43){\makebox(0,0)[l]{\footnotesize$0 $}}
\put(218,43){\makebox(0,0)[l]{\footnotesize$0 $}}
\put(268,43){\makebox(0,0)[l]{\footnotesize$0 $}}
\put(368,43){\makebox(0,0)[l]{\footnotesize$0 $}}

\put(270,43){\line(1,0){150}}
\put(270,43){\line(-1,0){150}}
\put(420,43){\line(0,1){180}}
\put(120,43){\line(0,1){180}}
\put(270,223){\line(1,0){150}}
\put(270,223){\line(-1,0){150}}
\put(320,43){\line(0,1){180}}
\put(220,43){\line(0,1){180}}
\put(270,103){\line(1,0){150}}
\put(270,103){\line(-1,0){150}}
\put(270,163){\line(-1,0){150}}
\put(270,163){\line(1,0){150}}

\put(130,0){\emph{Fig.1. The uninorm $U $ in Theorem \ref{th32}.}}
\end{picture}
\end{minipage}

\begin{minipage}{11pc}
\setlength{\unitlength}{0.75pt}\begin{picture}(600,220)
\put(266,37){$\bullet$}\put(267,30){\makebox(0,0)[l]{\footnotesize$0 $}}
\put(327,98){$\bullet$}\put(340,103){\makebox(0,0)[l]{\footnotesize$b $}}
\put(266,160){$\bullet$}\put(250,171){\makebox(0,0)[l]{\footnotesize$j $}}
\put(204,99){$\bullet$}\put(184,104){\makebox(0,0)[l]{\footnotesize$k $}}
\put(236,132){$\bullet$}\put(216,139){\makebox(0,0)[l]{\footnotesize$n $}}
\put(266,100){$\bullet$}\put(280,104){\makebox(0,0)[l]{\footnotesize$s $}}
\put(235,69){$\bullet$}\put(218,65){\makebox(0,0)[l]{\footnotesize$m $}}
\put(298,69){$\bullet$}\put(308,65){\makebox(0,0)[l]{\footnotesize$a $}}
\put(298,132){$\bullet$}\put(312,138){\makebox(0,0)[l]{\footnotesize$e $}}
\put(266,190){$\bullet$}\put(267,207){\makebox(0,0)[l]{\footnotesize$1 $}}

\put(270,41){\line(-1,1){61}}
\put(270,165){\line(0,1){30}}
\put(209,104){\line(1,1){61}}
\put(240,73){\line(1,1){30}}
\put(270,105){\line(-1,1){30}}
\put(270,41){\line(1,1){61}}
\put(331,105){\line(-1,1){60}}

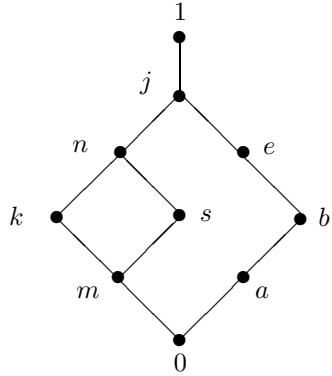
\put(170,-10){\emph{Fig.2. The bounded lattice $L_{1}$.}}
\end{picture}\
\end{minipage}

\begin{table}[htbp]
\centering
\caption{The closure operators $cl_{1} $ and $cl_{2} $ on $L_{1}$.}
\label{Tab:89}\

\begin{tabular}{c|c c c c c c c c c c}
\hline
  $x $ & $0 $ & $a $ & $b $ & $e $ & $m $ & $k $ & $s $ & $n $ & $j $ & $1 $ \\
\hline
  $cl_{1}(x )$ & $0 $ & $b $ & $b $ & $e $ & $k $ & $k $ & $n $ & $n $ & $j $ & $1 $ \\
\hline
  $cl_{2}(x )$ & $k $ & $j $ & $j $ & $j $ & $k $ & $k $ & $n $ & $n $ & $j $ & $1 $ \\
\hline
\end{tabular}
\end{table}

\begin{table}[htbp]
\centering
\caption{The uninorm $U  $ on $L_{1}$.}
\label{Tab:01}\

\begin{tabular}{c|c c c c c c c c c c}
\hline
  $U  $ & $0 $ & $a $ & $b $ & $e $ & $m $ & $k $ & $s $ & $n $ & $j $ & $1 $ \\
\hline
  $0 $ & $0 $ & $0 $ & $0 $ & $0 $ & $0 $ & $0 $ & $0 $ & $0 $ & $0 $ & $0 $ \\

  $a $ & $0 $ & $0 $ & $0 $ & $a $ & $0 $ & $0 $ & $0 $ & $0 $ & $b $ & $b $ \\

  $b $ & $0 $ & $0 $ & $0 $ & $b $ & $0 $ & $0 $ & $0 $ & $0 $ & $b $ & $b $ \\

  $e $ & $0 $ & $a $ & $b $ & $e $ & $m $ & $k $ & $s $ & $n $ & $j $ & $1 $ \\

  $m $ & $0 $ & $0 $ & $0 $ & $m $ & $0 $ & $0 $ & $0 $ & $0 $ & $k $ & $k $ \\

  $k $ & $0 $ & $0 $ & $0 $ & $k $ & $0 $ & $0 $ & $0 $ & $0 $ & $k $ & $k $ \\

  $s $ & $0 $ & $0 $ & $0 $ & $s $ & $0 $ & $0 $ & $0 $ & $0 $ & $n $ & $n $ \\

  $n $ & $0 $ & $0 $ & $0 $ & $n $ & $0 $ & $0 $ & $0 $ & $0 $ & $n $ & $n $ \\

  $j $ & $0 $ & $b $ & $b $ & $j $ & $k $ & $k $ & $n $ & $n $ & $j $ & $1 $ \\

  $1 $ & $0 $ & $b $ & $b $ & $1 $ & $k $ & $k $ & $n $ & $n $ & $1 $ & $1 $ \\
\hline
\end{tabular}
\end{table}\

 In Example \ref{ex1}, it is easily observed that $U (j ,a )=b \neq a $ for $a \in [0 ,e [$ and $j \in ]e ,1 ]$. So $U \notin \mathcal{U}_{min}^{*}\cup \mathcal{U}_{min}^{1}$. This implies that a uninorm $U$ constructed in Theorem \ref{th32}   does not necessarily belong to $ \mathcal{U}_{min}^{*}\cup \mathcal{U}_{min}^{1}$ in general. Now, we impose a requirement on $L$ to ensure that $U \in \mathcal{U}_{min}^{*}$.



\begin{proposition}\label{re156} Let $U $ be a uninorm in Theorem \ref{th32}.
 If $]0 ,e [\subseteq\{x \}$ for some $x\in L$, or if  $x\parallel y$ for  $x ,y \in ]0 ,e [$, then $ U \in \mathcal{U}_{min}^{*}$.
\end{proposition}
\begin{proof}
It is easily seen that $ U \in \mathcal{U}_{min}^{*}$ if $]0 ,e [=\emptyset$. Next, we consider two cases to show that $U \in \mathcal{U}_{min}^{*}$.

Case 1: $]0 ,e [=\{x \}$  for some $x\in L$.

Take any $x \in ]0 ,e [$.  On one hand, since $x \leq cl_{1}(x )$ and $x <x \vee e $, we have $x \leq cl_{1}(x )\wedge(x \vee e )$. On the other hand, since $cl_{1}(x )\notin [e ,1 ]$ and $x \vee e =e $, we have $ cl_{1}(x )\wedge(x \vee e )<e $.  Then $ cl_{1}(x )\wedge(x \vee e )\in [x ,e [$. However, $]x ,e [=\emptyset$. Thus, $ cl_{1}(x )\wedge(x \vee e )=x $. Therefore, $U (x ,y )=x $ for all $y \in ]e ,1 ]$. Moreover, by Theorem \ref{th32}, we know that $U (0 ,y )=0 $ for all $y \in ]e ,1 ]$. Thus, $ U \in \mathcal{U}_{min}^{*}$.

Case 2: $x \parallel y $ for  $x ,y \in ]0 ,e [$.

By Theorem \ref{th32}, we can obtain $U (x ,y )=cl_{1}(x )\wedge (x \vee e )$ for all $x \in ]0 ,e [$ and $y \in ]e ,1 ]$. Since $x <x \vee e  =e $, $x \leq cl_{1}(x )$ and $cl_{1}(x )\notin [e ,1 ]$, we have
$cl_{1}(x )\wedge (x \vee e )\in [x ,e [$. However, $]x ,e [=\emptyset$. Thus, $cl_{1}(x )\wedge (x \vee e )= x $. Therefore, for $x \in ]0 ,e [$, $U (x ,y )=x $ for all $y \in ]e ,1 ]$. Moreover, by Theorem \ref{th32}, we know that $U (0 ,y )=0 $ for all $y \in ]e ,1 ]$. Thus, $ U \in \mathcal{U}_{min}^{*}$.
\end{proof}

In particular, if $cl_{1}=cl_{2}$  in Theorem \ref{th32}, then we can derive a new construction method for uninorms using a single closure operator and a t-conorm.


\begin{proposition}\label{co122}
Let $e \in L\setminus\{0 ,1 \}$,
$S $ be a t-conorm on $[e ,1 ]$ and $cl$ be a closure operator on $L$. Define $U :L^{2}\rightarrow L$  as follows:
\begin{eqnarray*}
U (x ,y )=\left\{
\begin{array}{ll}
S (x , y ) & $if$\ (x ,y )\in [e ,1 ]^{2},\\
x  & $if$\ (x ,y )\in (I_{e } \cup ]0 ,e [)\times \{e \},\\
y  & $if$\ (x ,y )\in \{e \}\times   (I_{e }\cup ]0 ,e [),\\
cl(x )\wedge (x \vee e ) & $if$\ (x ,y )\in (I_{e }\cup ]0 ,e [)\times ]e ,1 ],\\
cl(y )\wedge (y \vee e ) & $if$\ (x ,y )\in ]e ,1 ]\times (I_{e }\cup ]0 ,e [),\\
0  &  otherwise.
\end{array} \right.
\end{eqnarray*}
Then $U $  is a uninorm on $L$ if and only if $cl(x )\notin [e ,1 ]$ for all $x \in I_{e }\cup ]0 ,e [$.
\end{proposition}


Moreover, if  $cl_{1}(x )=cl_{2}(x )=x$ for all $x \in L$ in Theorem \ref{th32}, that is,  if $cl(x) = x$ as stated in Proposition \ref{co122},  then we  obtain the existing result in the literature.


\begin{remark}\label{re3}

In Proposition \ref{co122}, if $cl(x )=x $ for all $x \in L$, then $cl(x )\wedge (x \vee e )=x $ and $cl(x )=x \notin [e ,1 ]$ for all $x \in I_{e }\cup ]0 ,e [$. In this case, the condition in Proposition \ref{co122} naturally holds. Then the uninorm $U$ constructed in Proposition \ref{co122} is exactly  the uninorm $U_{s}: L^{2}\rightarrow L$ constructed by Kara\c{c}al and Mesiar (\cite{FK15}, Theorem 1), which is defined as follows:
\begin{eqnarray*}
U_{s}(x ,y )=\left\{
\begin{array}{ll}
S (x ,y ) & $if$\ (x ,y )\in [e ,1 ]^{2},\\
x  & $if$\ (x ,y )\in (I_{e }\cup[0 ,e [)\times [e ,1 ],\\
y  & $if$\ (x ,y )\in [e ,1 ]\times(I_{e }\cup [0 ,e [),\\
0  & otherwise.
\end{array} \right.
\end{eqnarray*}
\end{remark}

\begin{remark}\label{co123}

If we take $cl_{1}(x )=x $ for all $x \in L$ in Theorem \ref{th32},
then we can obtain another new  construction method of uninorms by a closure operator and a t-conorm as follow. Specifically, let $e \in L\setminus\{0 ,1 \}$,
$S $ be a t-conorm on $[e ,1 ]$ and $cl $ be a closure operator on $L$. Define $U :L^{2}\rightarrow L$  as follows:
\begin{eqnarray*}
U (x ,y )=\left\{
\begin{array}{ll}
S (x ,y ) & $if$\ (x ,y )\in [e ,1 ]^{2},\\
x  & $if$\ (x ,y )\in [0 ,e [\times [e ,1 ]\cup I_{e }\times \{e \},\\
y  & $if$\ (x ,y )\in [e ,1 ]\times  [0 ,e [\cup \{e \}\times I_{e },\\
cl(x )\wedge (x \vee e ) & $if$\ (x ,y )\in I_{e }\times  ]e ,1 ],\\
cl(y )\wedge (y \vee e ) & $if$\ (x ,y )\in ]e ,1 ]\times I_{e },\\
0  & otherwise.
\end{array} \right.
\end{eqnarray*}
Then $U $
is a uninorm on $L$ if and only if $cl(x )\notin [e ,1 ]$ for all $x \in I_{e }$.
\end{remark}

It is easy to see that the uninorm $U$ in Remark \ref{co123} belongs to $\mathcal{U}_{min}^{*}$. In the following example, we   show that the uninorm $U $ in Remark \ref{co123} does not belong to $\mathcal{U}_{min}\cup \mathcal{U}_{max}^{r}$.


\begin{example}\label{ex2}
Given a bounded lattice $L_{2}=\{0 ,a ,e ,b ,m ,k ,s ,n ,1 \}$ depicted in Fig.3., a t-conorm $S $ on $[e ,1 ]$ defined by $S (x ,y )=x \vee y $ for all $x ,y \in [e ,1 ]$ and a closure operator $cl $ on $L_{2}$ defined by $cl(x )=x \vee k $ for all $x \in L_{2}$.
Then the closure operator $cl $ satisfies the condition in Remark \ref{co123}, i.e., $cl(x )\notin [e ,1 ]$ for all $x \in I_{e }$. Using the construction method in Remark \ref{co123}, we can obtain a uninorm $U :L_{2}^{2}\rightarrow L_{2}$ with the neutral element $e$, as shown in Table \ref{Tab:05}.  It is easy to see that  $U (b ,s )=n \neq s $ and $U (a ,s )=0 \neq a $. So $U \notin \mathcal{U}_{min}\cup \mathcal{U}_{max}^{r}$.
\end{example}

\begin{minipage}{11pc}
\setlength{\unitlength}{0.75pt}\begin{picture}(600,190)
\put(266,37){$\bullet$}\put(267,30){\makebox(0,0)[l]{\footnotesize$0 $}}
\put(327,98){$\bullet$}\put(340,103){\makebox(0,0)[l]{\footnotesize$e $}}
\put(266,160){$\bullet$}\put(267,177){\makebox(0,0)[l]{\footnotesize$1 $}}
\put(204,99){$\bullet$}\put(185,104){\makebox(0,0)[l]{\footnotesize$k $}}
\put(236,132){$\bullet$}\put(218,139){\makebox(0,0)[l]{\footnotesize$n $}}
\put(266,100){$\bullet$}\put(280,104){\makebox(0,0)[l]{\footnotesize$s $}}
\put(235,69){$\bullet$}\put(222,65){\makebox(0,0)[l]{\footnotesize$m $}}
\put(298,69){$\bullet$}\put(308,65){\makebox(0,0)[l]{\footnotesize$a $}}
\put(298,132){$\bullet$}\put(312,138){\makebox(0,0)[l]{\footnotesize$b $}}

\put(270,41){\line(-1,1){61}}
\put(209,104){\line(1,1){61}}
\put(240,73){\line(1,1){30}}
\put(270,105){\line(-1,1){30}}
\put(270,41){\line(1,1){61}}
\put(331,105){\line(-1,1){60}}

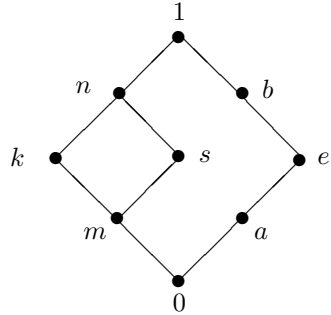
\put(170,-10){\emph{Fig.3. The bounded lattice $L_{2}$.}}
\end{picture}\
\end{minipage}\\

\begin{table}[htbp]
\centering
\caption{The uninorm $U  $ on $L_{2}$.}
\label{Tab:05}\

\begin{tabular}{c|c c c c c c c c c}
\hline
  $U  $ & $0 $ & $a $ & $e $ & $m $ & $k $ & $s $ & $n $ & $b $ & $1 $ \\
\hline
  $0 $ & $0 $ & $0 $ & $0 $ & $0 $ & $0 $ & $0 $ & $0 $ & $0 $ & $0 $ \\

  $a $ & $0 $ & $0 $ & $a $ & $0 $ & $0 $ & $0 $ & $0 $ & $a $ & $a $ \\

  $e $ & $0 $ & $a $ & $e $ & $m $ & $k $ & $s $ & $n $ & $b $ & $1 $ \\

  $m $ & $0 $ & $0 $ & $m $ & $0 $ & $0 $ & $0 $ & $0 $ & $k $ & $k $ \\

  $k $ & $0 $ & $0 $ & $k $ & $0 $ & $0 $ & $0 $ & $0 $ & $k $ & $k $ \\

  $s $ & $0 $ & $0 $ & $s $ & $0 $ & $0 $ & $0 $ & $0 $ & $n $ & $n $ \\

  $n $ & $0 $ & $0 $ & $n $ & $0 $ & $0 $ & $0 $ & $0 $ & $n $ & $n $ \\

  $b $ & $0 $ & $a $ & $b $ & $k $ & $k $ & $n $ & $n $ & $b $ & $1 $ \\

  $1 $ & $0 $ & $a $ & $1 $ & $k $ & $k $ & $n $ & $n $ & $1 $ & $1 $ \\
\hline
\end{tabular}
\end{table}\

Dually, we   present a
method to construct uninorms using two interior operators and a t-norm
on bounded lattices,  without proof.

\begin{theorem}\label{th31}
Let $e \in L\setminus\{0 ,1 \}$,
$T $ be a t-norm on $[0 ,e ]$, and $int_{1}$ and $int_{2}$ be interior operators on $L$ with $int_{2}(x)\leq int_{1}(x)$ for all $x \in L\setminus[0 ,e ]$. Define $U :L^{2}\rightarrow L$ as follows:
\begin{eqnarray*}
U (x ,y )=\left\{
\begin{array}{ll}
T (x , y ) & $if$\ (x ,y )\in [0 ,e ]^{2},\\
x  & $if$\ (x ,y )\in (I_{e }\cup ]e ,1 [)\times \{e \},\\
y  & $if$\ (x ,y )\in \{e \}\times (I_{e }\cup ]e ,1 [),\\
int_{2}(x )\vee (x \wedge e ) & $if$\ (x ,y )\in I_{e } \times[0 ,e [,\\
int_{2}(y )\vee (y \wedge e ) & $if$\ (x ,y )\in [0 ,e [ \times I_{e },\\
int_{1}(x )\vee (x \wedge e ) & $if$\ (x ,y )\in ]e ,1 [  \times [0 ,e [,\\
int_{1}(y )\vee (y \wedge e ) & $if$\ (x ,y )\in [0 ,e [   \times ]e ,1 [,\\
1  & otherwise.
\end{array} \right.
\end{eqnarray*}
Then $U $ is a uninorm on $L$ if and only if $int_{2}(x )\notin [0 ,e ]$ for all $x \in I_{e }$ and $int_{1}(x )\notin [0 ,e ]$ for all $x \in ]e ,1 [$.
\end{theorem}

\begin{remark}\label{rema22}
Notice that the uninorm $U :L^{2}\rightarrow L$ in Theorem \ref{th31} can also be defined such that
\begin{eqnarray*}
U (x ,y )=\left\{
\begin{array}{ll}
T (x , y ) & $if$\  (x ,y )\in [0 ,e [^{2},\\
e  & $if$\  (x ,y )\in \{e \}\times \{e \},\\
x  & $if$\  (x ,y )\in (L\setminus\{e ,1 \})\times \{e \},\\
y  & $if$\  (x ,y )\in \{e \}\times (L\setminus\{e ,1 \}),\\
int_{2}(x )\vee (x \wedge e ) & $if$\  (x ,y )\in I_{e }\times [0 ,e [,\\
int_{2}(y )\vee (y \wedge e ) & $if$\  (x ,y )\in [0 ,e [\times I_{e },\\
int_{1}(x )\vee (x \wedge e ) & $if$\  (x ,y )\in ]e ,1 [\times [0 ,e [,\\
int_{1}(y )\vee (y \wedge e ) & $if$\  (x ,y )\in [0 ,e [\times ]e ,1 [,\\
1  & $if$\  (x ,y )\in \{1 \}\times L\cup L\times\{1 \}\cup I_{e }\times I_{e } \cup (I_{e }\cup ]e ,1 [)\times]e ,1 [ \\
 &\cup ]e ,1 [\times(I_{e }\cup ]e ,1 [).
\end{array} \right.
\end{eqnarray*}
\end{remark}

\begin{remark}
From Remark \ref{rema22}, the structure of the uninorm $U$ is illustrated in Fig.4., where $ \oplus$ denotes $int_{1}(x )\vee (x \wedge e ), \ominus$ denotes $int_{1}(y )\vee (y \wedge e ), \otimes$ denotes $int_{2}(x )\vee (x \wedge e )$ and $ \oslash$ denotes $int_{2}(y )\vee (y \wedge e ).$
\end{remark}

\begin{minipage}{11pc}
\setlength{\unitlength}{0.75pt}\begin{picture}(600,230)
\put(158,25){\makebox(0,0)[l]{\footnotesize$[0 ,e [$}}
\put(72,73){\makebox(0,0)[l]{\footnotesize$[0 ,e [$}}
\put(72,130){\makebox(0,0)[l]{\footnotesize$]e ,1 [$}}
\put(256,25){\makebox(0,0)[l]{\footnotesize$]e ,1 [$}}
\put(209,25){\makebox(0,0)[l]{\footnotesize$\{e \}$}}
\put(366,25){\makebox(0,0)[l]{\footnotesize$I_{e }$}}
\put(80,103){\makebox(0,0)[l]{\footnotesize$\{e \}$}}
\put(90,193){\makebox(0,0)[l]{\footnotesize$I_{e }$}}
\put(147,73){\makebox(0,0)[l]{\footnotesize$T (x ,y )$}}
\put(168,101){\makebox(0,0)[l]{\footnotesize$x $}}
\put(168,132){\makebox(0,0)[l]{\footnotesize$\ominus$}}
\put(168,193){\makebox(0,0)[l]{\footnotesize$\oslash$}}
\put(218,73){\makebox(0,0)[l]{\footnotesize$y $}}
\put(218,101){\makebox(0,0)[l]{\footnotesize$e $}}
\put(218,132){\makebox(0,0)[l]{\footnotesize$y $}}
\put(218,193){\makebox(0,0)[l]{\footnotesize$y $}}
\put(268,73){\makebox(0,0)[l]{\footnotesize$\oplus$}}
\put(268,101){\makebox(0,0)[l]{\footnotesize$x $}}
\put(268,132){\makebox(0,0)[l]{\footnotesize$1 $}}
\put(268,193){\makebox(0,0)[l]{\footnotesize$1 $}}
\put(368,73){\makebox(0,0)[l]{\footnotesize$\otimes$}}
\put(368,101){\makebox(0,0)[l]{\footnotesize$x $}}
\put(368,132){\makebox(0,0)[l]{\footnotesize$1 $}}
\put(368,193){\makebox(0,0)[l]{\footnotesize$1 $}}
\put(80,163){\makebox(0,0)[l]{\footnotesize$\{1 \}$}}
\put(310,25){\makebox(0,0)[l]{\footnotesize$\{1 \}$}}
\put(318,73){\makebox(0,0)[l]{\footnotesize$1 $}}
\put(318,101){\makebox(0,0)[l]{\footnotesize$1 $}}
\put(318,132){\makebox(0,0)[l]{\footnotesize$1 $}}
\put(318,193){\makebox(0,0)[l]{\footnotesize$1 $}}
\put(168,163){\makebox(0,0)[l]{\footnotesize$1 $}}
\put(218,163){\makebox(0,0)[l]{\footnotesize$1 $}}
\put(268,163){\makebox(0,0)[l]{\footnotesize$1 $}}
\put(318,163){\makebox(0,0)[l]{\footnotesize$1 $}}
\put(368,163){\makebox(0,0)[l]{\footnotesize$1 $}}

\put(270,43){\line(1,0){150}}
\put(270,43){\line(-1,0){150}}
\put(420,43){\line(0,1){180}}
\put(120,43){\line(0,1){180}}
\put(270,223){\line(1,0){150}}
\put(270,223){\line(-1,0){150}}
\put(320,43){\line(0,1){180}}
\put(220,43){\line(0,1){180}}
\put(270,103){\line(1,0){150}}
\put(270,103){\line(-1,0){150}}
\put(270,163){\line(-1,0){150}}
\put(270,163){\line(1,0){150}}

\put(135,0){\emph{Fig.4. The uninorm $U $ in Theorem \ref{th31}.}}
\end{picture}
\end{minipage}\\

It can be seen that $U$ in  Theorem \ref{th31}  does not necessarily  belong to $\mathcal{U}_{max}^{*}\cup \mathcal{U}_{max}^{0}$.
However, the uninorm $U$ in Theorem \ref{th31} can belong to $ \mathcal{U}_{max}^{*}$ when certain special conditions are imposed on the bounded lattice $L$, as shown in the following proposition.

\begin{proposition}\label{rema158}
Let $U $ be a uninorm in  Theorem \ref{th31}. If $]e ,1 [\subseteq\{x \}$ for some $x\in L$, or $x \parallel y $ for all $x ,y \in ]e ,1 [$, then $U \in \mathcal{U}_{max}^{*}$.
\end{proposition}



As a dual result of Proposition \ref{co122}, we present the following result.


\begin{proposition}\label{co789}
Let $e\in L\setminus\{0 ,1 \}$,
$T$ be a t-norm on $[0 ,e ]$ and $int$ be an interior operator on $L$. Define $U :L^{2}\rightarrow L$ as follows:
\begin{eqnarray*}
U (x ,y )=\left\{
\begin{array}{ll}
T(x ,y ) & $if$\ (x ,y )\in [0 ,e ]^{2},\\
x  & $if$\ (x ,y )\in (I_{e }\cup ]e ,1 [)\times \{e \},\\
y  & $if$\ (x ,y )\in \{e \}\times (I_{e }\cup ]e ,1 [),\\
int(x )\vee (x \wedge e ) & $if$\ (x ,y )\in (I_{e }\cup ]e ,1 [)\times [0 ,e [,\\
int(y )\vee (y \wedge e ) & $if$\ (x ,y )\in [0 ,e [ \times (I_{e }\cup ]e ,1 [),\\
1  & otherwise.
\end{array} \right.
\end{eqnarray*}
Then  $U$ is a uninorm on $ L$ if and only if $int(x )\notin [0 ,e ]$ for all $x \in I_{e }\cup ]e ,1 [$.
\end{proposition}


\begin{remark}\label{re9}
In Proposition \ref{co789}, if $int(x )=x $ for all $x \in L$, then $int(x )\vee (x \wedge e )=x $ and $int(x )=x \notin [0 ,e ]$ for all $x \in I_{e }\cup ]e ,1 [$. In this case, the condition in Proposition \ref{co789} naturally holds. Then the uninorm $U$ constructed in Proposition \ref{co789} is consistent with  the uninorm $U_{t}: L^{2}\rightarrow L$ constructed by Kara\c{c}al and Mesiar (\cite{FK15}, Theorem 1), which is defined as follows
\begin{eqnarray*}
U_{t}(x ,y )=\left\{
\begin{array}{ll}
T (x ,y ) & $if$\ (x ,y )\in [0 ,e ]^{2},\\
x  & $if$\ (x ,y )\in (I_{e }\cup]e ,1 ])\times [0 ,e ],\\
y  & $if$\ (x ,y )\in [0 ,e ]\times (I_{e }\cup]e ,1 ]),\\
1  & otherwise.
\end{array} \right.
\end{eqnarray*}
\end{remark}


If $int_{1}(x) = x$ for all $x \in L$ in Theorem \ref{th31}, then we obtain another method to construct uninorms using an interior operator and a t-norm on a bounded lattice, which can be considered as the dual result of Remark \ref{co123}.

\begin{proposition}\label{co156}
Let $e \in L\setminus\{0 ,1 \}$,
$T $ be a t-norm on $[0 ,e ]$ and $int $ be an interior operator on $L$. Define $U :L^{2}\rightarrow L$ as follows:
\begin{eqnarray*}
U (x ,y )=\left\{
\begin{array}{ll}
T (x ,y ) & $if$\ (x ,y )\in [0 ,e ]^{2},\\
x  & $if$\ (x ,y )\in ]e ,1 ]\times  [0 ,e ]\cup I_{e }\times \{e \},\\
y  & $if$\ (x ,y )\in [0 ,e ]\times ]e ,1 ] \cup \{e \}\times I_{e },\\
int(x )\vee (x \wedge e ) & $if$\ (x ,y )\in I_{e }\times [0 ,e [,\\
int(y )\vee (y \wedge e ) & $if$\ (x ,y )\in [0 ,e [\times  I_{e },\\
1  & otherwise.
\end{array} \right.
\end{eqnarray*}
Then $U $  is a uninorm on $L$ if and only if $int(x )\notin [0 ,e ]$ for all $x \in I_{e }$.
\end{proposition}

\begin{remark}\label{rem189}
Let $U $ be a uninorm constructed in  Proposition \ref{co156}. Then $U \in \mathcal{U}_{max}^{*}$ but $U\notin \mathcal{U}_{max}\cup \mathcal{U}_{min}^{r}$.
\end{remark}

In the following theorem, we still use a t-conorm on $[e ,1 ]$ and two closure operators on $L$ to  construct a uninorm on $L$.


\begin{theorem}\label{th35}
Let  $ e \in L\setminus\{0 ,1 \}$, $S $ be a t-conorm on $[e ,1 ]$, and $cl_{1}$ and $cl_{2}$ be two closure operators on $L$ with $cl_{1}(x )\leq cl_{2}(x )$ for all $x \in L\setminus[e ,1 ]$. Define $U :L^{2}\rightarrow L$  as follows:
\begin{eqnarray*}
U (x ,y )=\left\{
\begin{array}{ll}
S (x , y ) & $if$\ (x ,y )\in [e ,1 [^{2},\\
x  & $if$\ (x ,y )\in ([0 ,e [\cup I_{e })\times \{e \},\\
y  & $if$\ (x ,y )\in \{e \}\times ([0 ,e [\cup I_{e }),\\
cl_{1}(x )\wedge (x \vee e ) & $if$\ (x ,y )\in ]0 ,e [\times ]e ,1 [,\\
cl_{1}(y )\wedge (y \vee e ) & $if$\ (x ,y )\in ]e ,1 [\times ]0 ,e [,\\
cl_{2}(x )\wedge (x \vee e ) & $if$\ (x ,y )\in I_{e }\times ]e ,1 [,\\
cl_{2}(y )\wedge (y \vee e ) & $if$\ (x ,y )\in ]e ,1 [\times I_{e },\\
1  & $if$\ (x ,y )\in L\times \{1 \}\cup \{1 \}\times L,\\
0  & otherwise.
\end{array} \right.
\end{eqnarray*}

$(1)$ Suppose that $cl_{1}(x )\notin [e ,1 ]$ for all $x \in ]0 ,e [$ and $cl_{2}(x )\notin [e ,1 ]$ for all $x \in I_{e }$. Then $U $ is a uninorm on $L$ if and only if $S (x ,y )<1 $ for all $x ,y \in ]e ,1 [$.

$(2)$ Suppose that $]e ,1 [\neq \emptyset$. Then $U $ is a uninorm on $L$ if and only if $S (x ,y )<1 $ for all $x ,y \in ]e ,1 [$, $cl_{1}(x )\notin [e ,1 ]$ for all $x \in ]0 ,e [$ and $cl_{2}(x )\notin [e ,1 ]$ for all $x \in I_{e }$.
\end{theorem}
\begin{proof}
$(1)$ Necessity: Assume that there exist $x ,y \in ]e ,1 [$ such that $S (x ,y )=1 $. Then $U (0 ,U (x ,y ))=U (0 ,S(x ,y ))=U (0 ,1 )=1 \neq 0 =U (0 ,y )$ $=U (U (0 ,x ),y )$. This contradicts the associativity of $U $.

Sufficiency:
By the definition of $U $, it is easy to obtain that  $U $ is commutative  and  $e $  is the neutral element of $U $.  Then it remains to prove the increasingness and the associativity of $U$. Taking into account Theorem \ref{th32}, it is sufficient to check the cases that differ from those in Theorem \ref{th32}.
For the increasingness, it is obvious that $U (x ,z )\leq U (y ,z )$ if $1 \in \{x ,y ,z \}$. For the associativity, it is easy to see that $U (x , U (y ,z ))=U (U (x ,y ),z )=1$ if $1 \in \{x ,y ,z \}$. By Proposition \ref{pro2.1}, we obtain that $U (x ,U (y ,z ))=U (U (x ,y ),z )$ for all $x ,y ,$ $z \in L$. Therefore, $U $ is a uninorm on $L$ with the neutral element $e $.

$(2)$ By (1), the sufficiency is straightforward. For the necessity, assume that $U$ is a uninorm. Then we will prove the following three conclusions:

(i) $S (x ,y )<1 $ for all $x ,y \in ]e ,1 [$. It follows immediately from (1).

(ii) $cl_{1}(x )\notin [e ,1 ]$ for all $x \in ]0 ,e [$. Take any $y \in ]e ,1 [$. Assume that there exists $x \in ]0 ,e [$ such that $cl_{1}(x )\in [e ,1 ]$. Then $U (x ,U (x ,y ))=U (x ,cl_{1}(x )\wedge (x \vee e ))=U (x ,e )=x $ and $U (U (x ,x ),y )=U (0 ,y )=0 $. This contradicts the associativity of $U $, as desired.

(iii)  $cl_{2}(x )\notin [e ,1 ]$ for all $x \in I_{e }$.  By Definition \ref{de29},  we know that $cl_{2}(x )\neq e $ for all $x \in I_{e }$. In the following, we just prove that $cl_{2}(x )\notin ]e ,1 ]$ for all $x \in I_{e }$. Take any $y \in ]e ,1 [$. Assume that there exists $x \in I_{e }$ such that $cl_{2}(x )\in ]e ,1 ]$. Then $U (x ,U (x ,y ))=U (x ,cl_{2}(x )\wedge (x \vee e ))=cl_{2}(x )\wedge (x \vee e )$ and $U (U (x ,x ),y )=U (0 ,y )=0 $. This contradicts the associativity of $U $, as desired.
\end{proof}

\begin{remark}
From Theorem \ref{th35}, the structure of the uninorm $U :L^{2}\rightarrow L$ is illustrated in Fig.5..
\end{remark}

\begin{minipage}{11pc}
\setlength{\unitlength}{0.75pt}\begin{picture}(600,220)
\put(158,25){\makebox(0,0)[l]{\footnotesize$]0 ,e [$}}
\put(72,73){\makebox(0,0)[l]{\footnotesize$]0 ,e [$}}
\put(72,130){\makebox(0,0)[l]{\footnotesize$]e ,1 [$}}
\put(258,25){\makebox(0,0)[l]{\footnotesize$]e ,1 [$}}
\put(210,25){\makebox(0,0)[l]{\footnotesize$\{e \}$}}
\put(311,25){\makebox(0,0)[l]{\footnotesize$\{1 \}$}}
\put(366,25){\makebox(0,0)[l]{\footnotesize$I_{e }$}}
\put(80,103){\makebox(0,0)[l]{\footnotesize$\{e \}$}}
\put(80,43){\makebox(0,0)[l]{\footnotesize$\{0 \}$}}
\put(109,25){\makebox(0,0)[l]{\footnotesize$\{0 \}$}}
\put(80,163){\makebox(0,0)[l]{\footnotesize$\{1 \}$}}
\put(90,193){\makebox(0,0)[l]{\footnotesize$I_{e }$}}
\put(168,73){\makebox(0,0)[l]{\footnotesize$0 $}}
\put(168,103){\makebox(0,0)[l]{\footnotesize$x $}}
\put(168,130){\makebox(0,0)[l]{\footnotesize$\ast$}}
\put(168,163){\makebox(0,0)[l]{\footnotesize$1 $}}
\put(168,193){\makebox(0,0)[l]{\footnotesize$0 $}}
\put(118,73){\makebox(0,0)[l]{\footnotesize$0 $}}
\put(118,103){\makebox(0,0)[l]{\footnotesize$0 $}}
\put(118,130){\makebox(0,0)[l]{\footnotesize$0 $}}
\put(118,163){\makebox(0,0)[l]{\footnotesize$1 $}}
\put(118,193){\makebox(0,0)[l]{\footnotesize$0 $}}
\put(218,73){\makebox(0,0)[l]{\footnotesize$y $}}
\put(218,103){\makebox(0,0)[l]{\footnotesize$e $}}
\put(218,130){\makebox(0,0)[l]{\footnotesize$y $}}
\put(218,163){\makebox(0,0)[l]{\footnotesize$y $}}
\put(218,193){\makebox(0,0)[l]{\footnotesize$y $}}
\put(268,73){\makebox(0,0)[l]{\footnotesize$\star$}}
\put(268,103){\makebox(0,0)[l]{\footnotesize$x $}}
\put(253,130){\makebox(0,0)[l]{\footnotesize$S (x ,y )$}}
\put(268,163){\makebox(0,0)[l]{\footnotesize$1 $}}
\put(268,193){\makebox(0,0)[l]{\footnotesize$\bullet$}}
\put(318,73){\makebox(0,0)[l]{\footnotesize$1 $}}
\put(318,103){\makebox(0,0)[l]{\footnotesize$x $}}
\put(318,130){\makebox(0,0)[l]{\footnotesize$1 $}}
\put(318,163){\makebox(0,0)[l]{\footnotesize$1 $}}
\put(318,193){\makebox(0,0)[l]{\footnotesize$1 $}}
\put(368,73){\makebox(0,0)[l]{\footnotesize$0 $}}
\put(368,103){\makebox(0,0)[l]{\footnotesize$x $}}
\put(368,130){\makebox(0,0)[l]{\footnotesize$\circ$}}
\put(368,163){\makebox(0,0)[l]{\footnotesize$1 $}}
\put(368,193){\makebox(0,0)[l]{\footnotesize$0 $}}

\put(118,43){\makebox(0,0)[l]{\footnotesize$0 $}}
\put(168,43){\makebox(0,0)[l]{\footnotesize$0 $}}
\put(218,43){\makebox(0,0)[l]{\footnotesize$0 $}}
\put(268,43){\makebox(0,0)[l]{\footnotesize$0 $}}
\put(318,43){\makebox(0,0)[l]{\footnotesize$1 $}}
\put(368,43){\makebox(0,0)[l]{\footnotesize$0 $}}

\put(270,43){\line(1,0){150}}
\put(270,43){\line(-1,0){150}}
\put(420,43){\line(0,1){180}}
\put(120,43){\line(0,1){180}}
\put(270,223){\line(1,0){150}}
\put(270,223){\line(-1,0){150}}
\put(320,43){\line(0,1){180}}
\put(220,43){\line(0,1){180}}
\put(270,103){\line(1,0){150}}
\put(270,103){\line(-1,0){150}}
\put(270,163){\line(-1,0){150}}
\put(270,163){\line(1,0){150}}

\put(130,0){\emph{Fig.5. The uninorm $U $ in Theorem \ref{th35}.}}
\end{picture}
\end{minipage}\\

The next example illustrates the construction method of uninorms on bounded lattices in Theorem \ref{th35}.

\begin{example}\label{ex1}
Given a bounded lattice $L_{3}=\{0 ,l ,m ,n ,r ,a ,b ,c ,e ,t ,$ $1 \}$ depicted in Fig.6., a $t$-conorm $S $ on $[e ,1 ]$ defined by $S (x ,y )=x \vee y $ for all $x ,y \in [e ,1 ]$ and two closure operators $cl_{1}$ and $cl_{2}$ shown in Table \ref{Tab:10}.
It is easy to see that the closure operators $cl_{1}$ and $cl_{2}$ on $L_{3}$ satisfy the conditions in Theorem \ref{th35}, i.e.,  $cl_{1}(x )\leq cl_{2}(x )$ for all $x \in L_{3}$, $cl_{1}(x )\notin [e ,1 ]$ for all $x \in ]0 ,e [$ and $cl_{2}(x ) \notin [e ,1 ]$ for all $x \in I_{e }$.
Using the construction method in Theorem \ref{th35}, we can obtain a uninorm $U :L_{3}^{2}\rightarrow L_{3}$ with the neutral element $e$, as  shown in Table \ref{Tab:02}.

\end{example}

In Table \ref{Tab:02} of Example \ref{ex1}, we can see that $U (t ,b )=c \neq b $ for $(b ,t )\in I_{e }\times ]e ,1 [$ and $U (r ,t )=a \neq r $ for $(r ,t )\in ]0 ,e [\times ]e ,1 [$.  This demonstrates that $U \notin \mathcal{U}_{min}^{*}\cup \mathcal{U}_{min}^{1}$.
However, $U \in \mathcal{U}_{min}^{1}$ when $L$ is equipped with certain requirements, as shown in the following proposition.

\begin{proposition}
Let $U $ be a uninorm in Theorem  \ref{th35}.
If $]0 ,e [\subseteq\{x \}$ and $I_{e }\subseteq\{y \}$, or if  $x \parallel y $ for $x ,y \in ]0 ,e [$ and $x \parallel y $ for $x ,y \in I_{e }$, then $U \in  \mathcal{U}_{min}^{1}$.
\end{proposition}
\begin{proof}
It can be immediately proved by the proof similar to Proposition \ref{re156}.
\end{proof}\

\begin{minipage}{11pc}
\setlength{\unitlength}{0.75pt}\begin{picture}(600,220)
\put(266,47){$\bullet$}\put(267,40){\makebox(0,0)[l]{\footnotesize$0 $}}
\put(266,170){$\bullet$}\put(250,181){\makebox(0,0)[l]{\footnotesize$c $}}
\put(236,108){$\bullet$}\put(212,114){\makebox(0,0)[l]{\footnotesize$m $}}
\put(236,142){$\bullet$}\put(215,149){\makebox(0,0)[l]{\footnotesize$n $}}
\put(266,108){$\bullet$}\put(280,113){\makebox(0,0)[l]{\footnotesize$a $}}
\put(266,142){$\bullet$}\put(280,149){\makebox(0,0)[l]{\footnotesize$b $}}
\put(235,79){$\bullet$}\put(215,77){\makebox(0,0)[l]{\footnotesize$l $}}
\put(266,79){$\bullet$}\put(280,79){\makebox(0,0)[l]{\footnotesize$r $}}
\put(296,142){$\bullet$}\put(312,148){\makebox(0,0)[l]{\footnotesize$e $}}
\put(296,170){$\bullet$}\put(312,181){\makebox(0,0)[l]{\footnotesize$t $}}
\put(266,200){$\bullet$}\put(267,217){\makebox(0,0)[l]{\footnotesize$1 $}}

\put(270,51){\line(-1,1){30}}
\put(270,51){\line(0,1){150}}
\put(240,145){\line(1,1){30}}
\put(240,83){\line(0,1){60}}
\put(300,145){\line(0,1){30}}
\put(270,114){\line(1,1){30}}
\put(300,175){\line(-1,1){30}}

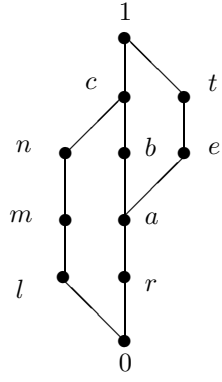
\put(160,10){\emph{Fig.6. The bounded lattice $L_{3}$.}}
\end{picture}\
\end{minipage}

\begin{table}[htbp]
\centering
\caption{The closure operators $cl_{1}$ and $cl_{2}$ on $L_{3}$.}
\label{Tab:10}\

\begin{tabular}{c|c c c c c c c c c c c c}
\hline
  $x $ & $0 $ & $r $ & $a $ & $e $ & $l $ & $m $ & $n $ & $b $ & $c $ & $t $ & $1 $ \\
\hline
  $cl_{1}(x )$ & $0 $ & $a $ & $a $ & $e $ & $n $ & $n $ & $n $ & $c $ & $c $ & $t $ & $1 $ \\
\hline
  $cl_{2}(x )$ & $a $ & $a $ & $a $ & $e $ & $c $ & $c $ & $c $ & $c $ & $c $ & $t $ & $1 $ \\
\hline
\end{tabular}
\end{table}\

\begin{table}[htbp]
\centering
\caption{The uninorm $U  $ on $L_{3}$.}
\label{Tab:02}\

\begin{tabular}{c|c c c c c c c c c c c c}
\hline
  $U  $ & $0 $ & $r $ & $a $ & $e $ & $l $ & $m $ & $n $ & $b $ & $c $ & $t $ & $1 $ \\
\hline
  $0 $ & $0 $ & $0 $ & $0 $ & $0 $ & $0 $ & $0 $ & $0 $ & $0 $ & $0 $ & $0 $ & $1 $ \\

  $r $ & $0 $ & $0 $ & $0 $ & $r $ & $0 $ & $0 $ & $0 $ & $0 $ & $0 $ & $a $ & $1 $ \\

  $a $ & $0 $ & $0 $ & $0 $ & $a $ & $0 $ & $0 $ & $0 $ & $0 $ & $0 $ & $a $ & $1 $ \\

  $e $ & $0 $ & $r $ & $a $ & $e $ & $l $ & $m $ & $n $ & $b $ & $c $ & $t $ & $1 $ \\

  $l $ & $0 $ & $0 $ & $0 $ & $l $ & $0 $ & $0 $ & $0 $ & $0 $ & $0 $ & $c $ & $1 $ \\

  $m $ & $0 $ & $0 $ & $0 $ & $m $ & $0 $ & $0 $ & $0 $ & $0 $ & $0 $ & $c $ & $1 $ \\

  $n $ & $0 $ & $0 $ & $0 $ & $n $ & $0 $ & $0 $ & $0 $ & $0 $ & $0 $ & $c $ & $1 $ \\

  $b $ & $0 $ & $0 $ & $0 $ & $b $ & $0 $ & $0 $ & $0 $ & $0 $ & $0 $ & $c $ & $1 $ \\

  $c $ & $0 $ & $0 $ & $0 $ & $c $ & $0 $ & $0 $ & $0 $ & $0 $ & $0 $ & $c $ & $1 $ \\

  $t $ & $0 $ & $a $ & $a $ & $t $ & $c $ & $c $ & $c $ & $c $ & $c $ & $t $ & $1 $ \\

  $1 $ & $1 $ & $1 $ & $1 $ & $1 $ & $1 $ & $1 $ & $1 $ & $1 $ & $1 $ & $1 $ & $1 $ \\
\hline
\end{tabular}
\end{table}




If $cl_{1} = cl_{2}$ in Theorem \ref{th35}, then we obtain a new construction method for uninorms using a closure operator and a t-conorm.
\begin{proposition}\label{rr0}
Let  $ e\in L\setminus\{0 ,1 \}$, $S $ be a t-conorm on $[e ,1 ]$ and $cl$ be a closure operator on $L$. Define $U :L^{2}\rightarrow L$  as follows:
\begin{eqnarray*}
U (x ,y )=\left\{
\begin{array}{ll}
S (x ,y ) & $if$\ (x ,y )\in [e ,1 [^{2},\\
x  & $if$\ (x ,y )\in ([0 ,e [\cup I_{e })\times \{e \},\\
y  & $if$\ (x ,y )\in \{e \}\times ([0 ,e [\cup I_{e }),\\
cl(x )\wedge (x \vee e ) & $if$\ (x ,y )\in (]0 ,e [\cup I_{e })\times ]e ,1 [,\\
cl(y )\wedge (y \vee e ) & $if$\ (x ,y )\in ]e ,1 [\times (]0 ,e [\cup I_{e }),\\
1  & $if$\ (x ,y )\in L\times \{1 \}\cup \{1 \}\times L,\\
0  &  otherwise.
\end{array} \right.
\end{eqnarray*}

$(1)$ Suppose that $cl(x )\notin [e ,1 ]$ for all $x \in ]0 ,e [\cup I_{e }$. Then $U$ is a uninorm on $L$ if and only if $S (x ,y )<1 $ for all $x ,y \in ]e ,1 [$.

$(2)$ Suppose that $]e ,1 [\neq \emptyset$. Then $U$ is a uninorm on $L$ if and only if $S(x ,y )<1 $ for all $x ,y \in ]e ,1 [$, $cl(x )\notin [e ,1 ]$ for all $x \in ]0 ,e [\cup I_{e }$.
\end{proposition}

If $cl_{1}(x)=x $ for all $x \in L$ in Theorem \ref{th35}, then we   obtain the following result.

\begin{proposition}\label{qq1}
Let  $ e \in L\setminus\{0 ,1 \}$, $S $ be a t-conorm on $[e ,1 ]$ and $cl $ be a closure operator on $L$. Define $U :L^{2}\rightarrow L$  as follows:
\begin{eqnarray*}
U (x ,y )=\left\{
\begin{array}{ll}
S (x ,y ) & $if$\ (x ,y )\in [e ,1 [^{2},\\
x  & $if$\ (x ,y )\in ([0 ,e [\cup I_{e })\times \{e \}\cup ]0 ,e [\times ]e ,1 [,\\
y  & $if$\ (x ,y )\in \{e \}\times ([0 ,e [\cup I_{e })\cup]e ,1 [\times ]e ,1 [,\\
cl(x )\wedge (x \vee e ) & $if$\ (x ,y )\in I_{e} \times ]e ,1 [,\\
cl(y )\wedge (y \vee e ) & $if$\ (x ,y )\in ]e ,1 [\times I_{e },\\
1  & $if$\ (x ,y )\in L\times \{1 \}\cup \{1 \}\times L,\\
0  & otherwise.
\end{array} \right.
\end{eqnarray*}

$(1)$ Suppose that $cl (x )\notin [e ,1 ]$ for all $x \in I_{e }$. Then $U $ is a uninorm on $L$ if and only if $S (x ,y )<1 $ for all $x ,y \in ]e ,1 [$.

$(2)$ Suppose that $]e ,1 [\neq \emptyset$. Then $U $ is a uninorm on $L$ if and only if $S (x ,y )<1 $ for all $x ,y \in ]e ,1 [$ and $cl(x )\notin [e ,1 ]$ for all $x \in I_{e }$.
\end{proposition}


Next,   we  present the dual version of Theorem \ref{th35} without proof.

\begin{theorem}\label{th36}
Let  $ e  \in L\setminus\{0 ,1 \}$, $T$ be a t-norm on $[0 ,e ]$, and $int_{1}$ and $int_{2}$ be two interior operators on $L$ with $int_{2}\leq int_{1}$ for all $x \in L\setminus[0 ,e ]$. Define $U :L^{2}\rightarrow L$  as follows:
\begin{eqnarray*}
U (x ,y )=\left\{
\begin{array}{ll}
T (x , y ) & $if$\ (x ,y )\in ]0 ,e ]^{2},\\
x  & $if$\ (x ,y )\in (I_{e }\cup ]e ,1 ])\times \{e \},\\
y  & $if$\ (x ,y )\in \{e \}\times (I_{e }\cup]e ,1 ]),\\
int_{1}(x )\vee (x \wedge e ) & $if$\ (x ,y )\in ]e ,1 [\times ]0 ,e [,\\
int_{1}(y )\vee (y \wedge e ) & $if$\ (x ,y )\in ]0 ,e [\times ]e ,1 [,\\
int_{2}(x )\vee (x \wedge e ) & $if$\ (x ,y )\in I_{e }\times ]0 ,e [,\\
int_{2}(y )\vee (y \wedge e ) & $if$\ (x ,y )\in ]0 ,e [\times I_{e },\\
0  & $if$\ (x ,y )\in L\times \{0 \}\cup \{0 \}\times L,\\
1  & otherwise.
\end{array} \right.
\end{eqnarray*}

$(1)$ Suppose that  $int_{2}(x )\notin [0 ,e ]$ for all $x \in I_{e }$ and $int_{1}(x )\notin [0 ,e ]$ for all $x \in ]e ,1 [$. Then $U $ is a uninorm on $L$ if and only if $0 <T (x ,y )$ for all $x ,y \in ]0 ,e [$.

$(2)$ Suppose that $]0 ,e [\neq \emptyset$. Then $U $ is a uninorm on $L$ if and only if $0 <T (x ,y )$ for all $x ,y \in ]0 ,e [$, $int_{2}(x )\notin [0 ,e ]$ for all $x \in I_{e }$ and $int_{1}(x )\notin [0 ,e ]$ for all $x \in ]e ,1 [$.
\end{theorem}

\begin{remark}
From Theorem \ref{th36}, the structure of the uninorm $U :L^{2}\rightarrow L$ is illustrated in Fig.7..
\end{remark}

\begin{minipage}{11pc}
\setlength{\unitlength}{0.75pt}\begin{picture}(600,220)
\put(158,25){\makebox(0,0)[l]{\footnotesize$]0 ,e [$}}
\put(72,73){\makebox(0,0)[l]{\footnotesize$]0 ,e [$}}
\put(72,130){\makebox(0,0)[l]{\footnotesize$]e ,1 [$}}
\put(258,25){\makebox(0,0)[l]{\footnotesize$]e ,1 [$}}
\put(212,25){\makebox(0,0)[l]{\footnotesize$\{e \}$}}
\put(80,43){\makebox(0,0)[l]{\footnotesize$\{0 \}$}}
\put(366,25){\makebox(0,0)[l]{\footnotesize$I_{e }$}}
\put(80,103){\makebox(0,0)[l]{\footnotesize$\{e \}$}}
\put(105,25){\makebox(0,0)[l]{\footnotesize$\{0 \}$}}
\put(90,193){\makebox(0,0)[l]{\footnotesize$I_{e }$}}
\put(168,43){\makebox(0,0)[l]{\footnotesize$0 $}}
\put(154,73){\makebox(0,0)[l]{\footnotesize$T (x ,y )$}}
\put(168,101){\makebox(0,0)[l]{\footnotesize$x $}}
\put(168,132){\makebox(0,0)[l]{\footnotesize$\ominus$}}
\put(168,193){\makebox(0,0)[l]{\footnotesize$\oslash$}}
\put(218,43){\makebox(0,0)[l]{\footnotesize$y $}}
\put(218,73){\makebox(0,0)[l]{\footnotesize$y $}}
\put(218,101){\makebox(0,0)[l]{\footnotesize$e $}}
\put(218,132){\makebox(0,0)[l]{\footnotesize$y $}}
\put(218,193){\makebox(0,0)[l]{\footnotesize$y $}}
\put(118,43){\makebox(0,0)[l]{\footnotesize$0 $}}
\put(118,73){\makebox(0,0)[l]{\footnotesize$0 $}}
\put(118,101){\makebox(0,0)[l]{\footnotesize$x $}}
\put(118,132){\makebox(0,0)[l]{\footnotesize$0 $}}
\put(118,193){\makebox(0,0)[l]{\footnotesize$0 $}}
\put(268,43){\makebox(0,0)[l]{\footnotesize$0 $}}
\put(268,73){\makebox(0,0)[l]{\footnotesize$\oplus$}}
\put(268,101){\makebox(0,0)[l]{\footnotesize$x $}}
\put(268,132){\makebox(0,0)[l]{\footnotesize$1 $}}
\put(268,193){\makebox(0,0)[l]{\footnotesize$1 $}}
\put(368,43){\makebox(0,0)[l]{\footnotesize$0 $}}
\put(368,73){\makebox(0,0)[l]{\footnotesize$\otimes$}}
\put(368,101){\makebox(0,0)[l]{\footnotesize$x $}}
\put(368,132){\makebox(0,0)[l]{\footnotesize$1 $}}
\put(368,193){\makebox(0,0)[l]{\footnotesize$1 $}}
\put(80,163){\makebox(0,0)[l]{\footnotesize$\{1 \}$}}
\put(310,25){\makebox(0,0)[l]{\footnotesize$\{1 \}$}}
\put(318,43){\makebox(0,0)[l]{\footnotesize$0 $}}
\put(318,73){\makebox(0,0)[l]{\footnotesize$1 $}}
\put(318,101){\makebox(0,0)[l]{\footnotesize$1 $}}
\put(318,132){\makebox(0,0)[l]{\footnotesize$1 $}}
\put(318,193){\makebox(0,0)[l]{\footnotesize$1 $}}

\put(118,163){\makebox(0,0)[l]{\footnotesize$0 $}}
\put(168,163){\makebox(0,0)[l]{\footnotesize$1 $}}
\put(218,163){\makebox(0,0)[l]{\footnotesize$1 $}}
\put(268,163){\makebox(0,0)[l]{\footnotesize$1 $}}
\put(318,163){\makebox(0,0)[l]{\footnotesize$1 $}}
\put(368,163){\makebox(0,0)[l]{\footnotesize$1 $}}

\put(270,43){\line(1,0){150}}
\put(270,43){\line(-1,0){150}}
\put(420,43){\line(0,1){180}}
\put(120,43){\line(0,1){180}}
\put(270,223){\line(1,0){150}}
\put(270,223){\line(-1,0){150}}
\put(320,43){\line(0,1){180}}
\put(220,43){\line(0,1){180}}
\put(270,103){\line(1,0){150}}
\put(270,103){\line(-1,0){150}}
\put(270,163){\line(-1,0){150}}
\put(270,163){\line(1,0){150}}

\put(140,0){\emph{Fig.7. The uninorm $U $ in Theorem \ref{th36}.}}
\end{picture}
\end{minipage}\\



 Similarly, the uninorm $U$ constructed in Theorem \ref{th36}  does not necessarily belong to $\mathcal{U}_{max}^{*}\cup \mathcal{U}_{max}^{0}$, in general.
However, $U \in \mathcal{U}_{max}^{0}$ when $L$ is equipped with certain requirements, as shown in the following proposition.

\begin{proposition}
Let $U $ be a uninorm in  Theorem \ref{th36}.
If $]e ,1 [\subseteq\{x \}$ and $I_{e }\subseteq\{y \}$, or if $x \parallel y $ for $x ,y \in ]e ,1 [$  and $x \parallel y $ for $x ,y \in I_{e } $, then $U \in \mathcal{U}_{max}^{0}$.
\end{proposition}


If  $int_{1}=int_{2}$ in Theorem \ref{th36}, then we  obtain the following proposition, which is exactly the dual result of Proposition \ref{rr0}.

\begin{proposition}\label{rr09}
Let  $ e \in L\setminus\{0 ,1 \}$, $T $ be a t-norm on $[0 ,e ]$ and $int $ be an interior operator on $L$. Define $U :L^{2}\rightarrow L$ as follows:
\begin{eqnarray*}
U(x ,y )=\left\{
\begin{array}{ll}
T (x ,y ) & $if$\ (x ,y )\in ]0 ,e ]^{2},\\
x  & $if$\ (x ,y )\in (I_{e }\cup ]e ,1 ])\times \{e \},\\
y  & $if$\ (x ,y )\in \{e \}\times (I_{e }\cup]e ,1 ]),\\
int(x )\vee (x \wedge e ) & $if$\ (x ,y )\in (I_{e }\cup]e ,1 [)\times ]0 ,e [,\\
int(y )\vee (y \wedge e ) & $if$\ (x ,y )\in ]0 ,e [\times (I_{e }\cup]e ,1 [),\\
0  & $if$\ (x ,y )\in L\times \{0 \}\cup \{0 \}\times L,\\
1  & otherwise.
\end{array} \right.
\end{eqnarray*}

$(1)$ Suppose that  $int(x )\notin [0 ,e ]$ for all $x \in I_{e }\cup ]e ,1 [$. Then $U $ is a uninorm on $L$ if and only if $0 <T (x ,y )$ for all $x ,y \in ]0 ,e [$.

$(2)$ Suppose that $]0 ,e [\neq \emptyset$. Then $U $ is a uninorm on $L$ if and only if $0 <T (x ,y )$ for all $x ,y \in ]0 ,e [$, $int(x )\notin [0 ,e ]$ for all $x \in I_{e }\cup  ]e ,1 [$.
\end{proposition}

If $int_{1}(x )=x $ for all $x \in L$ in Theorem \ref{th36}, then we obtain the dual result of Proposition \ref{qq1}.

\begin{proposition}\label{qq2}
Let  $ e \in L\setminus\{0 ,1 \}$, $T $ be a t-norm on $[0 ,e ]$ and $int $ be an interior operator on $L$. Define $U :L^{2}\rightarrow L$  as follows
\begin{eqnarray*}
U(x ,y )=\left\{
\begin{array}{ll}
T (x ,y ) & $if$\ (x ,y )\in ]0 ,e ]^{2},\\
x  & $if$\ (x ,y )\in (I_{e }\cup ]e ,1 ])\times \{e \}\cup ]e ,1 [\times ]0 ,e [,\\
y  & $if$\ (x ,y )\in \{e \}\times (I_{e }\cup]e ,1 ])\cup ]0 ,e [\times ]e ,1 [,\\
int(x )\vee (x \wedge e ) & $if$\ (x ,y )\in I_{e }\times ]0 ,e [,\\
int(y )\vee (y \wedge e ) & $if$\ (x ,y )\in ]0 ,e [\times I_{e },\\
0  & $if$\ (x ,y )\in L\times \{0 \}\cup \{0 \}\times L,\\
1  & otherwise.
\end{array} \right.
\end{eqnarray*}

$(1)$ Suppose that  $int(x )\notin [0 ,e ]$ for all $x \in I_{e }$. Then $U $ is a uninorm on $L$ if and only if $0 <T (x ,y )$ for all $x ,y \in ]0 ,e [$.

$(2)$ Suppose that $]0 ,e [\neq \emptyset$. Then $U $ is a uninorm on $L$ if and only if $0 <T (x ,y )$ for all $x ,y \in ]0 ,e [$ and $int(x )\notin [0 ,e ]$ for all $x \in I_{e }$.
\end{proposition}


\section{Conclusion}
\label{}


In this paper, we propose new construction methods for uninorms using two comparable closure operators (or, respectively, two comparable interior operators) on bounded lattices.
 It is worth noting that   the newly constructed uninorms need not belong to $\mathcal{U}_{max}^{*}\cup \mathcal{U}_{max}^{0}$ or $\mathcal{U}_{min}^{*}\cup \mathcal{U}_{min}^{1}$,  in general.  Moreover, we consider some special cases of the above results, some of which correspond exactly to the known construction methods.


 In this paper, we highlight the following points.
\begin{itemize}
\item  Two comparable  closure operators(or, respectively, two comparable interior operators) are used to construct uninorms on bounded lattices.

\item  Different from the uninorms  in  the literature, for  $x \in ]0 ,e [\cup I_{e }$ and $y \in]e ,1 [$ (or $x \in ]e ,1 [\cup I_{e }$ and $y \in]0 ,e [$),  we construct uninorms using closure operators (or  interior operators)  and show that the new uninorms need not  belong to  $\mathcal{U}_{min}^{*}\cup \mathcal{U}_{min}^{1}$ (or $\mathcal{U}_{max}^{*}\cup \mathcal{U}_{max}^{0}$).
\item  Additional constraints on two comparable closure operators (or two comparable interior operators)  are  both  sufficient and necessary.
\item For the construction  methods in Theorems  \ref{th32}, \ref{th31}, \ref{th35}(1) and \ref{th36}(1), there are no requirement on $L$; for the  construction methods in Theorem  \ref{th35}(2) ($]e ,1 [\neq \emptyset$) and in Theorem  \ref{th36}(2) ($]0 ,e [\neq \emptyset$), the requirements are weak.
\end{itemize}

Finally, we provide some explanations for Theorem \ref{th32}, which is the main result of this paper.
\begin{itemize}
\item
In  Theorem  \ref{th32},  it is easily seen that $U (x ,y )=0 $ for
$(x ,y )\in [0 ,e [\times [0 ,e [$.
This means that $U$ is the weakest triangular norm on the  region  $[0 ,e ]^{2}$.
 However, we have attempted to explore construction methods for uninorms by replacing $0$ with t-norms, t-conorms, or other operators.
In this case, in addition to the two closure operators, the results may require additional conditions beyond just the necessary and sufficient conditions on the closure operators.
So, we present  the result as Theorem  \ref{th32}.
In the future, we will try to    explore construction methods for uninorms, in which $U (x ,y )\neq 0 $ for $(x ,y )\in [0 ,e [\times [0 ,e [$.


\item  In  Theorem  \ref{th32},  the conditions  that   $cl_{1}(x )\notin [e ,1 ]$ for all $x \in ]0 ,e [$ and $cl_{2}(x )\notin [e ,1 ]$ for all $x \in I_{e }$ are both sufficient and necessary.
   To demonstrate the rationality of these conditions, we construct some examples to show that nontrivial closure operators on the given lattice satisfying the requirements do indeed exist (see Tables \ref{Tab:89} and \ref{Tab:10}).
\end{itemize}


In the future, we will attempt to apply our method to the setting of \cite{ZX23} and to other aggregation functions on a bounded lattice.
Moreover, as is well known, the construction methods of uninorms have been widely investigated from a theoretical perspective.
This motivates us to explore some potential applications of uninorms.

%
%

\end{document}